\documentclass[10pt,leqno]{amsart}
\usepackage{graphicx}

\usepackage{indentfirst,biblatex}
\bibliography{bibliography}
\theoremstyle{definition}
\usepackage{comment}
\usepackage{amssymb,amsthm,amsmath}
\usepackage{xcolor,paralist,hyperref,etoolbox,setspace}
\newtheorem{theorem}{Theorem}[]
\newtheorem{thm}[theorem]{Theorem}
\theoremstyle{definition}
\newtheorem{definition}[theorem]{Definition}
\newtheorem{dfn}[theorem]{Definition}

\newtheorem{lemma}[theorem]{Lemma}
\newtheorem{lem}[theorem]{Lemma}

\newtheorem{prop}[theorem]{Proposition}
\newtheorem{corollary}[theorem]{Corollary}
\newtheorem{cor}[theorem]{Corollary}

\newcommand{\ac}{\operatorname{AC}}
\newcommand{\inv}{^{-1}}
\newcommand{\res}{\upharpoonleft}
\newcommand{\s}{\mathcal}
\newcommand{\bb}{\mathbb}
\newcommand{\fix}{\operatorname{fix}}
\newcommand{\sym}{\operatorname{sym}}
\newcommand{\at}{\operatorname{at}}
\newcommand{\ip}[1]{\langle {#1} \rangle}
\newcommand{\lra}{\Leftrightarrow}
\newcommand{\im}{\operatorname{im}}
\newcommand{\dom}{\operatorname{dom}}
\newcommand{\id}{\operatorname{id}}
\newcommand{\aut}{\operatorname{aut}}
\newcommand{\supp}{\operatorname{supp}}

\newcommand{\zfa}{\text{ZFA}}
\newcommand{\zf}{\text{ZF}}

\newcommand{\sat}{\text{SAT}}
\newcommand{\maj}{\operatorname{maj}}
\newcommand{\pol}{\operatorname{Pol}}

\newcommand{\np}{\text{NP}}
\hypersetup{ colorlinks=true, linkcolor=black, filecolor=black, urlcolor=black }

\usepackage{tikz, tikz-cd}
\usepackage{todonotes}

\begin{document}
\title{Compactness principles for CSPs and the Axiom of Choice} 

\author[A Li]{Anthony Li}
\address{Independent researcher}
\email{ali4@andrew.cmu.edu}

\author[I Shah]{Ishin Shah}
\address{Carnegie Mellon University, Dept. of Mathematics, 5000 Forbes Ave, Pittsburgh, PA 15213}
\email{ipshah@andrew.cmu.edu}

\author[M Snodgrass]{Matthew Snodgrass}
\address{Carnegie Mellon University, Dept. of Mathematics, 5000 Forbes Ave, Pittsburgh, PA 15213}
\email{msnogra@andrew.cmu.edu}

\author[R Thornton]{Riley Thornton}
\address{University of Michigan, Dept. of Mathematics, 530 Church St, Ann Arbor, MI 48109}
\email{rileyjt@umich.edu}

\author[R Zhou]{Rui Zhou}
\address{UCLA, Dept. of Mathematics, 520 Portola Plaza, Los Angeles, CA 90095}
\email{ruizhou@math.ucla.edu}

\let\thefootnote\relax

\begin{abstract}

We can associate a compactness principle $\s K_{\s D}$ to the Constraint Satisfaction Problem (CSP) with constraint library $\mathcal{D}$. These principles were studied by Kat\'ay, T\'oth, and Vidny\'anszky and by Rorabaugh, Tardif, and Wehlau. We expand their work comparing the strength of $K_{\mathcal{D}}$ for varying structures $\mathcal{D}$. We characterize the structures whose compactness principles are provable from $\zf$; these turn out to be the width-1 structures. We compare some important compactness principles, namely those of $2\sat$, $3\text{LIN}2$, and $K_2$, settling a question of Kat\'ay, T\'oth, and Vidny\'anszky. And, we find an infinite chain and an infinite anti-chain of compactness principles related to directed cycles and finite choice.
\end{abstract} 

\bigskip
\maketitle
\section{Introduction}
Constraint Satisfaction Problems (CSPs) are a well-studied class of problems originating from computer science. Given a finite relational structure $\mathcal{D}$, which we think of as the constraint library, the corresponding problem $\text{CSP}(\mathcal{D})$ is the problem of deciding whether there is a homomorphism from a structure $\mathcal{X}$ (an instance) to $\mathcal{D}$. Some examples include:
\begin{itemize}
    \item For $k\sat$, the structure on $\{0,1\}$ with all $k$-ary relations, $\text{CSP}(k\sat)$ is equivalent the satisfaction problem for Boolean formulae in CNF with $k$-ary disjuncts of literals.
    \item For the complete graph $K_n$, $\text{CSP}(K_n)$ is the graph $n$-coloring problem
    \item  For $\bb F_p$ the structure on $\bb Z/p\bb Z$ with all affine relations, $\text{CSP}(\mathbb{F}_p)$ is the problem of solving linear equations over the field of order $p$
\end{itemize}
(More precise definitions are given in Section \ref{sec: notation}). CSPs can be ordered by pp-constructibility, a combinatorial notion of reduction (see Section \ref{sec: Tools}), and the pp-constructibility relation on CSPs is in Galois correspondence with clones of polymorphisms under various algebraic comparisons. The upshot of this is that most complexity classes of CSPs can be characterized using algebraic methods. For example, $\text{CSP}(\mathcal{D})$ is solvable by linear relaxation if and only if $\mathcal{D}$ has a symmetric polymorphism. For a survey of algorithmic applications, see \cite{DagstuhlSurvey}.\par
A handful of authors have recently applied these tools in set theory. For example, the fourth author observed that most pp-constructions also yield Borel reductions in descriptive set theory \cite{thornton2022algebraicapproachborelcsps}. Most relevant to this work, Rorabaugh, Tardif, and Wehlau and, independently, Kat\'ay, T\'oth, and Vidny\'anszky introduced a family of compactness principles associated to CSPs \cite{kátay2023cspdichotomyaxiomchoice}\cite{RorabaughTardifWehlau}.

\begin{definition}
    For a constraint library $\mathcal{D}$, $K_{\mathcal{D}}$ is the statement that ``For every instance $\mathcal{X}$ of $\mathcal{D}$, if every finite substructure of $\mathcal{X}$ has a homomorphism to $\mathcal{D}$, then $\mathcal{X}$ has a homomorphism to $\mathcal{D}$".
\end{definition}

One can show that $\zf+UFL$ (the Ultrafilter Lemma) implies $K_\mathcal{D}$ for every finite constraint library $\mathcal{D}$. Kat\'ay, T\'oth and Vidny\'anszky characterize the CSPs whose compactness principals have maximal strength
\begin{theorem}[\cite{kátay2023cspdichotomyaxiomchoice}]
    For any finite relational structure $\mathcal{D}$, if $\mathcal{D}$ has no Taylor polymorphism, then $\zf+K_{\mathcal{D}}\vdash\text{UFL}$. Conversely, assuming $\zf$ is consistent, if $\mathcal{D}$ has a Taylor polymorphism then $\zf+K_{\mathcal{D}}\not\vdash\text{UFL}$.
\end{theorem} Bulatov and Zhuk independently showed that the same algebraic criterion characterizes polynomial-time solvable CSPs \cite{Bulatov}\cite{Zhuk}. Thus, if $\text{P}\neq\np$ and $\zf$ is consistent, then $\zf+K_{\mathcal{D}}$ precisely when $\text{CSP}(\mathcal{D})$ is NP-complete.

In this paper, we expand on this work and explore the structure of the provability order for compactness principles. We construct infinite chains and anti-chains, characterize the provable instances of compactness, and nearly classify the compactness principles for Boolean structures.
\subsection{Results}
First, we use a result of the fourth author and the arc-consistency algorithm to characterize compactness principles of minimum strength:
\begin{theorem}[\zf]
    For a finite relational structure $\s D$, $\zf$ proves $K_{\mathcal{D}}$ if and only if $\mathcal{D}$ admits totally symmetric polymorphisms of arbitrarily large arity.
\end{theorem}

See Section \ref{sec: Tools} for the relevant definitions. This result has also been announced by Tardif, though his proof relies on unpublished work of Solovay, and it is unclear what cardinal hypotheses he requires. As observed in \cite{kátay2023cspdichotomyaxiomchoice}, if $\mathcal{D}\leq_{\text{pp}}\mathcal{E}$ then $\zf+K_{\mathcal{E}}\vdash K_{\mathcal{D}}$; using this result, along with the algebraic characterization of pp-definability and classic work of Post on clones of Boolean functions, we find there are at most 5 classes of $K_\mathcal{D}$ for Boolean $\mathcal{D}$:
\begin{prop}[\zf]
    If $|\mathcal{D}|=2$, then $K_\mathcal{D}$ is equivalent to one of the following:
    \begin{enumerate}
        \item $K_{\varnothing}\equiv \top$
        \item $K_{K_2}\equiv \ac(2)$
        \item $K_{\mathbb{F}_2}$
        \item $K_{2\sat}$
        \item $K_{3\sat}\equiv \text{UFL}$
    \end{enumerate} Where $\ac(2)$ is the axiom of choice for pairs.
\end{prop}
Working in \zf, we have the following diagram of implications:

\begin{center} \begin{tikzcd}
    & & \arrow[dl] K_{2\sat} &\\
    K_{\varnothing} & \arrow[l] K_{K_2} & & \arrow[ul] \arrow[dl] \text{UFL}\\
    & & \arrow[ul] K_{\mathbb{F}_2} &
\end{tikzcd} \end{center}

If $\zf$ is consistent, then the arrow on the left and the arrows on the right cannot be reversed. Kat\'ay, T\'oth, and Vidny\'anszky asked in \cite[5.2, p. 17]{kátay2023cspdichotomyaxiomchoice} whether any of the two arrows in the middle can be reversed. We show that they cannot, and further that all of these principles are distinct:
\begin{theorem}
    If $\zf$ is consistent, then $\zf+K_{K_2}$ does not prove $K_{\mathbb{F}_2}$ or $K_{2\sat}$, and $\zf +K_{2\sat}$ does not prove $K_{\mathbb{F}_2}$.
\end{theorem} The only question remaining for Boolean structures is whether $K_{\bb F_2}$ is strictly stronger than $K_{2\sat}$.

Finally, we construct an infinite chain and an infinite anti-chain of compactness principles related to directed cycles and finite choice. Write $C_n$ for the directed cycle of length $n$ (see \ref{sec: notation} for a precise definition). For anti-chains, we have a very strong sense of incomparability:
\begin{theorem}[Con(\zf)]
   For each prime $p$ there is a model $M_p$ of \zf+$\neg K_{C_p}$ such that whenever $q$ is a prime distinct from p, $M_p\models K_{C_q}$.
\end{theorem}
So for each $p$, we can fix a single model $M_p$ witnessing the incomparability of $K_{C_p}$ with all other $K_{C_q}$. By taking the disjoint union of directed cycles, we can then form structures whose compactness principles have increasing implication strength:
\begin{definition}
    For $n\geq 2$, $\mathcal{D}_n$ is the constraint library $(D,R,U_2,...,U_n)$ with domain $D=\{c^2_0,c^2_1,...,c^n_0,...,c^n_{n-1}\}$, a binary relation $R=\{(c^{i}_{j}, c^{i}_{j+1})|2\leq i\leq n,j\in\mathbb{Z}/i\mathbb{Z}\}$, and $n-1$ unary relations $U_k=\{c^k_j:j\in\mathbb{Z}/i\mathbb{Z}\}$.
\end{definition}
So the $c^i_j $’s are vertices of a digraph, and $R$ is the edge relation, and the $Ui_i$’s are the auxiliary predicates for controlling which cycle a component of an instance can map into.
\begin{definition}
    For $n<\omega$, $\ac(n)$ is the statement that every family of sets of size at most $n$ has a choice function. 
\end{definition}
\begin{lemma}
    The following are equivalent over \zf:
    \begin{enumerate}
        \item $K_{\mathcal{D}_n}$
        \item $\ac(n)$
        \item $\bigwedge_{i=2}^nK_{C_p}$
    \end{enumerate}
\end{lemma}
\begin{cor}
    The principles $K_{\s D_p}$ for $p$ a prime form an infinite chain in the provability order.
\end{cor}

For ease of exposition, the main body of the paper will work under ZFA (see Section \ref{sec: Tools} for definitions). We translate our proofs to work under $\zf$ via standard forcing argument in Appendix \ref{appendix: zf}.

\subsection{Acknowledgements} The authors thank Claude Tardif and Zoltan Vidny\'anszky for helpful conversations about the topic of this paper. Much of this research was
carried out as part of Carnegie Mellon University’s SEMS program; we also thank
Irina Gheorghiciuc for organizing the program. The third author was supported by
NSF MSPRF grant DMS-2202827

\subsection{Notation} \label{sec: notation}
A \textbf{relational structure} $\mathcal{X}$ is a tuple $(X,R_1^{\mathcal{X}},\dots,R_n^{\mathcal{X}},...)$ where $X$ is a set and the $R_i$ are relations on $X$. We say $X$ is the \textbf{domain} of $\mathcal{X}$. For all $1\leq i\leq n$, let $a_i$ be the arity of $R_i$. The \textbf{signature} of $\mathcal{X}$ is the tuple $(a_1,\dots,a_n,...)$. We will only be concerned with finite arity relations and thus assume all arities $a_i$ are finite going forward. When $X$ is finite and there are only finitely many relations, we also call this a \textbf{constraint library} and usually denote it by $\mathcal{D}=(D,R_1,\dots,R_n)$. An \textbf{instance} of a constraint library $\mathcal{D}=(D,R_1,\dots,R_n)$ is a relational structure $\mathcal{X}=(X,R_1^{\mathcal{X}},\dots,R_n^{\mathcal{X}})$ such that for all $1\leq i\leq n$, $R_i^{\mathcal{X}}$ is a relation on $X$ with the same arity as $R_i$. In other words, it is a relational structure with the same signature as some given constraint library.

Given a set $\mathcal{F}$ of function symbols, a \textbf{term} is either a variable symbol $x_i$ or of the form $f(t_1,...,t_n)$ where $f\in\mathcal{F}$ has arity $n$ and each $t_n$ is a term. The \textbf{height} of a term is the depth of the recursion on term complexity. A formula is a height $n$ identity if it is of the form $t_1=t_2$ where $t_1,t_2$ has height $n$. For example, $f(x,y,y)=f(x,x,z)$ is a height one identity, but $f(x,y)=x$ is not.

Given two relational structures $\mathcal{X}=(X,R_1^{\mathcal{X}},\dots,R_n^{\mathcal{X}})$ and $\mathcal{Y}=(Y,R_1^{\mathcal{Y}},\dots,R_n^{\mathcal{Y}})$ with the same signature, a \textbf{homomorphism} from $\mathcal{X}$ to $\mathcal{Y}$ is a function $f:X\to Y$ such that for all $1\leq i\leq n$ and all $(x_1,\dots,x_{a_i})\in R_i^{\mathcal{X}}$, $(f(x_1),\dots,f(x_{a_i}))\in R_i^{\mathcal{Y}}$. A \textbf{solution} to an instance $\mathcal{X}$ of $\mathcal{D}$ is a homomorphism from $\mathcal{X}$ to $\mathcal{D}$. $\text{CSP}(\mathcal{D})$ is the problem of determining whether there is a solution to a given instance of $\mathcal{D}$.

Let $\mathcal{X}=(X,R_1,\dots,R_n)$ be a relational structure with signature $(a_1,\dots,a_n)$ and $A\subseteq X$. The \textbf{induced substructure} on $A$ is the relational structure $(A,R_1\cap A^{a_i},\dots,R_n\cap A^{a_n})$. A \textbf{finite substructure} of $\mathcal{X}$ is the induced substructure on some finite subset of the domain of $\mathcal{X}$. Note that if we are thinking of $\mathcal{X}$ as an instance of some constraint library, then every induced substructure is an instance of the same constraint library.\par

When $\mathcal{D}=(D,R_1,\dots,R_n)$ is a constraint library, the \textbf{compactness principle} associated to $\mathcal{D}$, denoted as $K_{\mathcal{D}}$, is the statement: for every instance $\mathcal{X}$ of $\mathcal{D}$, if every finite substructure of $\mathcal{X}$ admits a solution (as instances of $\mathcal{D}$), then $\mathcal{X}$ admits a solution.

Note that $K_{\mathcal{D}}$ follows from ZFC. In this paper we study these compactness principles under weaker forms of choice. 

Let $\mathcal{X}_1,\dots,\mathcal{X}_n$ be relational structures all with the same signature $(a_1,\dots,a_n)$, and set $\mathcal{X}_i=(X_i,R_1^i,\dots,R_n^i)$. The product $\prod_{i=1}^n\mathcal{X}_i$ of these relational structures is another relational structure $\mathcal{X}=(X,R_1,\dots,R_n)$ with the same signature called the \textbf{categorical product} where
\begin{enumerate}
    \item the domain $X$ is the product $\prod_{i=1}^nX_i$ of the domains of the $\mathcal{X}_i$;
    \item $((x_{11},\dots,x_{1n}),\dots,(x_{a_i1},\dots,x_{a_in}))\in R_i\Leftrightarrow\forall1\leq j\leq n$, $(x_{1j},\dots,x_{a_ij})\in R_i^j$.
\end{enumerate} In other words, a relation in the categorical product holds if and only if the relation holds for each coordinate. 

The strength of $K_{\mathcal{D}}$ is controlled by the polymorphism algebra of $\mathcal{D}$, which we are now ready to define:
\begin{definition}
    A \textbf{polymorphism} $f$ of $\mathcal{X}$ is a homomorphism $\mathcal{X}^n\to\mathcal{X}$ for some $n$, where we let $\mathcal{X}^n$ denote $\Pi^n_{i=1}\mathcal{X}$. The set of all polymorphisms for a constraint library $\mathcal{D}$ is denoted by $\pol(\mathcal{D})$.
\end{definition}

A handful of structures will be important throughout this article.

\begin{definition}
    For each $n$, $K_n$ is the complete graph on $n$ vertices:
    \[K_n=(\{0,..., n-1\}, \not=).\] And, $C_n$ is the directed $n$-cycle:
    \[C_n=(\bb Z/n\bb Z, \{(x,x+1):x\in \bb Z/n\bb Z\}).\]
    
    We write  $\mathbb{F}_2$ for the structure whose domain is $\mathbb{Z}/2\mathbb{Z}$ and whose relations are all affine relations on $\bb Z/2\bb Z$. We write $3\text{LIN}2$ for the reduct of $\mathbb{F}_2$ to ternary relations.

    Let $2\sat$ be the constraint library $(\{0,1\},R_{\rightarrow},R_{\neg}, True)$ where
    \begin{enumerate}
        \item $R_{\neg}(x,y)\;:\lra\; x\not=y$ (or equivalently $x=1-y$) 
        \item $R_{\rightarrow}(x,y)\;:\lra x\leq y$
        \item $True$ is of arity $1$ and $True=\{1\}$.
    \end{enumerate}

    For $k\geq 2$, we let $k\sat$ be the structure on $\{0,1\}$ equipped with all $k$-ary relations.
\end{definition} 

Note that $K_2$ is the same as the directed $2$-cycle. The structures $\bb F_2$ and $3\text{LIN}2$ turn out to have equivalent compactness principles. The theoretical advantage of $3\text{LIN}2$ is that it has only finitely many relations. Also, we have a chosen a convenient relational basis for $2\sat$ (we could have done so for $k\sat$ generally, but we do not need to). We will justify this choice in the next section. 

\subsection{Tools} \label{sec: Tools}
As in K{\'a}tay, T{\'o}th, and Vidny{\'a}nszky \cite{kátay2023cspdichotomyaxiomchoice}, the main tools we use come from the theory of pp-construction and the theory of symmetric models. 

\begin{definition}
    A structure $\mathcal{D}$ = $(D, R_1, ..., R_n)$ \textbf{positive primitive defines} or
\textbf{pp-defines} a relation $R$ of arity $a$ if we have
\[R(x_1, ..., x_a) \;\Longleftrightarrow\; (\exists z_1, ..., z_b)
\bigwedge^{m}_{i=1}
\alpha_i(x_1, ..., x_a, z_1, ..., z_b),
\] where each $\alpha_i$
is of the form $R_{\ell}(\pi_S(x_1, ..., x_a, z_1, ..., z_b))$ for some $\ell$, $S$, or of the form
$x_{\ell} = z_{\ell}'$. We call such an expression a \textbf{pp-definition} of $R$ over $\mathcal{D}$. Say that $\mathcal{D}$ pp-defines a structure $\mathcal{E}$ if $\mathcal{E}$ has the same domain as $\mathcal{D}$ and $\s D$ pp-defines every
relation of $\mathcal{E}$.
\end{definition}
In words, $\mathcal{D}$ pp-defines $\mathcal{E}$ if we can write any relation in $\mathcal{E}$ as a conjunction of equalities and relations from $\mathcal{D}$ among the inputs and some dummy variables. For example, $3\text{LIN}2$ pp-defines $\bb F_2$:
\[x_1+x_2+x_3+x_4=1 \;\lra\; (\exists z)\;z+x_1+x_2=0 \wedge z+x_3+x_4=1.\] And, $2\sat$ pp-defines every binary relation on $\{0,1\}$.

Given a pp-definition of $\s E$ in $\s D$, we can translate instances of $\mathcal{E}$ into instances of $\mathcal{D}$ in a uniform way by replacing the constraints of $\mathcal{E}$ with gadgets of constraints of $\mathcal{D}$, so that a solution to the latter is exactly equivalent to a solution to the former.

When working with constraints on possibly different domains, we can reduce $\mathcal{E}$-instances into $\mathcal{D}$-instances if the domain and the relations of $\mathcal{E}$ are pp-definable over powers of $\mathcal{D}$ modulo a pp-definable equivalence; in this case, we can use tuples of variables instead of just variables to build our gadgets. More formally,
\begin{definition}
     A \textbf{pp-definable substructure} of $\s D$ is the induced substructure on some pp-definable subset of the domain. We say $\mathcal{D}$ \textbf{pp-interprets} $\mathcal{E}$ if there is $n$ such that $\mathcal{E}$ is a quotient of a pp-definable substructure of $\mathcal{D}^n$ by a pp-definable equivalence relation.
\end{definition}
Homomorphisms also allow us to translate problems with different domains, which motivates the following definition:
\begin{definition}
    We say that $\mathcal{D}$ and $\mathcal{E}$ are hom-equivalent if there are homomorphisms $f : \mathcal{D} \rightarrow \mathcal{E}$ and $g : \mathcal{E} \rightarrow \mathcal{D}$. A constraint library $\mathcal{D}$ is a \textbf{core} if every homomorphism $f : \mathcal{D} \rightarrow \mathcal{D}$ is an automorphism.
\end{definition}
If $\s D$ is a core, then it turns that out $\s D$ pp-defines the orbits of each of its elements under automorphisms. Thus, any instance of a single unary expansion of $\s D$ can also be coded as an instance of $\s D$. We have many reasonable ways for $\mathcal{D}$ to simulate $\mathcal{E}$. The definition of pp-construction collects these.
\begin{definition}
    $\mathcal{D}$ \textbf{pp-constructs} $\mathcal{E}$ if there is a sequence of structures $\mathcal{E}_0, \mathcal{E}_1, ..., \mathcal{E}_n$
so that $\mathcal{D} = \mathcal{E}_0, \mathcal{E} = \mathcal{E}_n$ and for every $i$ one of the following holds:
\begin{enumerate}
    \item $E_{i+1}$ is pp-interpretable in $E_i$
    \item $E_i$ is homomorphically equivalent to $E_{i+1}$
    \item $E_i$ is a core and $E_{i+1}$ is $E_i$ expanded by a singleton unary relation $U(x) \Longleftrightarrow x = a$
\end{enumerate}
In such scenario, we write $\mathcal{E}\leq_{\text{pp}}\mathcal{D}$. If $\mathcal{E}\leq_{\text{pp}}\mathcal{D}$ and $\mathcal{D}\leq_{\text{pp}}\mathcal{E}$, we way that $\s D$ and $\s E$ are pp-equivalent and write $\s D\equiv_{\text{pp}}\s E$.
\end{definition}

The definition of pp-construction is useful for our study of compactness principles. Kat\'ay, T\'oth, and Vidny\'anszky used pp-definitions to prove implications among the $K_{\s D}$:
\begin{theorem}[\cite{kátay2023cspdichotomyaxiomchoice}]
\label{tool3}
    If $\mathcal{D}$ pp-constructs $\mathcal{E}$, then $K_{\mathcal{D}}$ implies $K_{\mathcal{E}}$ over \zf.
\end{theorem}
Moreover, as we mentioned earlier, there is an algebraic correspondence between $\leq_{\text{pp}}$ and polymorphisms; this is made precise by the following theorems, which allow us to decide whether or not $\mathcal{D}\leq_{\text{pp}}\mathcal{E}$ by studying their polymorphisms:
\begin{theorem}
\label{tool1}
For constraint libraries $\mathcal{D}$ and $\mathcal{E}$ with the same domain, $\mathcal{D}$ pp-defines $\mathcal{E}$ if and only if $\pol(\mathcal{D}) \subseteq \pol(\mathcal{E})$, and
$\mathcal{D}$ pp-constructs $\mathcal{E}$ if and only if any height 1 identity satisfied by $\pol(\mathcal{D})$ is satisfied by some subset of $\pol(\mathcal{E})$
\end{theorem}
\begin{corollary}
\label{tool2}
    If $\pol(\mathcal{D}) \subseteq \pol(\mathcal{E})$, then $K_{\mathcal{D}}$ implies $K_{\mathcal{E}}$ over \zf.
\end{corollary}

The height one identity defining totally symmetric polymorphisms will be of particular importance in this paper.

\begin{dfn}
    Say that a polymorphism $f$ is \textbf{totally symmetric}, or $TS$, if it only depends on its set of inputs, without regard to order or multiplicity. More formally, $f$ is TS if
    \[f(x_1,..., x_n)=f(y_1,..., y_n)\] whenever $\{x_1,..., x_n\}=\{y_1,..., y_n\}.$
\end{dfn}

Another main tool we use is the permutation model construction. This construction only gives models of ZFA, set theory with atoms. We work over this weaker system since the construction captures the combinatorial core of our arguments. We are able transfer these results to $\zf$ as well using standard forcing arguments; see Appendix \ref{appendix: zf} for details. 
\begin{definition}
    ZFA denotes the usual Zermelo-Fraenkel (ZF) axiomatic system with the additional axiom
$$(\exists A) \;\bigl(A\text{ is infinite}\wedge (\forall a\in A, x) \;(x\not\in a)\bigr)$$
and with the axiom of extensionality restricted to sets not in $A$. The set $A$ will be referred to as the set of atoms.
\end{definition}

Given a model $\mathcal{M}$ of ZFA with the set of atoms $A$, and $G\in \mathcal{M}$ a group of permutations of $A$, we can construct a submodel $HS_G$ of ZFA as follows:\par
We abuse notation and extend each $\sigma\in G$ to automorphism of $\mathcal{M}$ defined by rank induction $$\sigma(X)=\{\sigma(x):x\in X\}$$
\begin{definition}
    If $E\subseteq A$, then $$\text{fix}(E)=\{\sigma\in G:\sigma\restriction E=\text{id}\}$$ is the set of pointwise stabilizers of $E$, where id is the identity function.
    For $X\in\mathcal{M}$, we write $$\text{sym}(X)=\{\pi\in G:\pi(X)=X\}$$ to be the setwise stabilizers of $X$. Let $\mathcal{F}$ be the filter on $G$ given by $$\mathcal{F}=\{S\subseteq G:(\exists E\subseteq A)\; E\text{ is finite}\wedge\text{fix}(E)\subseteq S\}$$
\end{definition}

To underscore the difference, observe that $\fix(X)=\bigcap\{\sym(x): x\in X\}.$ 

\begin{definition} $X$ is \textbf{symmetric} when $\text{sym}(X)\in\mathcal{F}$, and $X$ is \textbf{hereditarily symmetric} (\textbf{HS}) if every element of the transitive closure of $\{X\}$ is symmetric.
\end{definition}
\begin{theorem}[\cite{jech2008axiom}]
    If $\mathcal{M}$ is a model of $\zfa$ and $G\in \mathcal{M}$ is a group permuting $A$, then the hereditarily symmetric model derived from $G$, HS$_G:=\{S\in\mathcal{M}:S\text{ is HS}\}$, is a model of $\zfa$
\end{theorem}
Note that we use a restricted definition of HS here, since we always use the filter generated by $\{\{X:X \supseteq \mathcal{F}\}: \mathcal{F} \subseteq A \text{ is finite}\}$.

Since HS$_G$ is a model of $\zfa$ for every model $\mathcal{M}$ and $G\in\mathcal{M}$, we can conclude:
\begin{theorem}
\label{tool4}
If ZFC proves that there is a group $G$ so that every $\text{HS}_G$ instance of $\mathcal{D}$ with a solution has a $\text{HS}_G$ solution but that there is an $\text{HS}_G$ instance of $\mathcal{E}$ with a solution but no $\text{HS}_G$ solution, then $K_{\mathcal{D}}$ does not imply $K_{\mathcal{E}}$ over $\zfa$.
\end{theorem}
\begin{proof}
    Note that if $X\in\text{HS}_G$ is an instance of $\mathcal{D}$, then every finite substructure $F$ of $X$ is also in $\text{HS}_G$, and a solution for $F$ is also in $\text{HS}_G$,
    \begin{align*}
        \text{HS}_G\models \text{every finite }F\subseteq X\text{ has a solution}&\Leftrightarrow \text{every finite }F\subseteq X\text{ has a solution}\\
        &\Leftrightarrow X\text{ has a solution}\\
        &\Leftrightarrow\exists f\in\text{HS}_G(f\text{ is a solution for }X)\\
        &\Leftrightarrow\text{HS}_G\models f\text{ is a solution for }X
    \end{align*}
    where the second equivalence follows by compactness in $V$, so $\text{HS}_G$ is a model of $\zfa+K_{\mathcal{D}}$. On the other hand, if $Y$ is the instance of $\mathcal{E}$ as in the theorem, then every finite substructure and its solution are also in $\text{HS}_G$, but $\text{HS}_G\models\text{ there is no solution for }Y$, so $\text{HS}_G\models\zfa+K_{\mathcal{D}}+\neg K_{\mathcal{E}}$.
\end{proof}

\section{2-element structures}
In this section, we make some progress on classifying compactness principles for structures on 2 elements (also called Boolean structures). Post classified all polymorphism algebras on 2 elements. These polymorphism algebras form a lattice when ordered by containment, which is shown in in Figure \ref{fig:postlattice}. We call this lattice Post's lattice. 

Figure \ref{fig:postlattice} follows the naming convention used by Wikipedia \cite{wikipediaPosts}. The main algebras of importance in this article are \[\mathsf{\Lambda}P=\ip{\wedge},\; UP_0=\ip{0}, \; UP_1=\ip{1}, \;\mathsf{V}P=\ip{\vee}, \;UD=\ip{\neg},\] \[AD=\ip{x+y+z, \neg}, \; AP=\ip{x+y+z},\; DM=\ip{\maj},\; D=\ip{\maj, \neg},\] \[DP=\ip{\maj, x+y+z},\; \top=\bigcup_i 2^{2^i}, \mbox{ and }\perp=\ip{}.\] Here, for a set of operations $A$, $\ip{A}$ means the smallest set of functions on $\{0,1\}$ closed under compositions and containing $A$ along with all of the projections. 

By Theorem \ref{tool1}, Post's lattice is equivalent to the lattice of 2-element structures ordered by the reverse of pp-definability. However, the relations between the compactness principles of structures only depend on the structures' pp-constructibility relations, which are generally much coarser. Let us group the elements in Post's lattice together accordingly.

\begin{figure}
    \centering
    \includegraphics[width=0.5\linewidth]{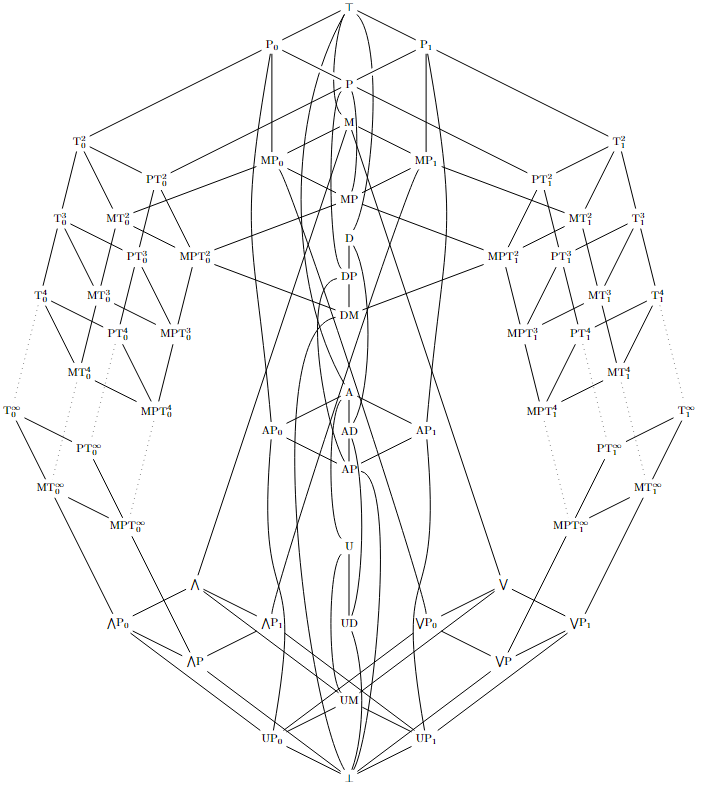}
    \caption{Post's lattice on clones for 2-element sets, ordered by containment}
    \label{fig:postlattice}
\end{figure}

\begin{figure}
    \centering
    \includegraphics[width=0.7\linewidth]{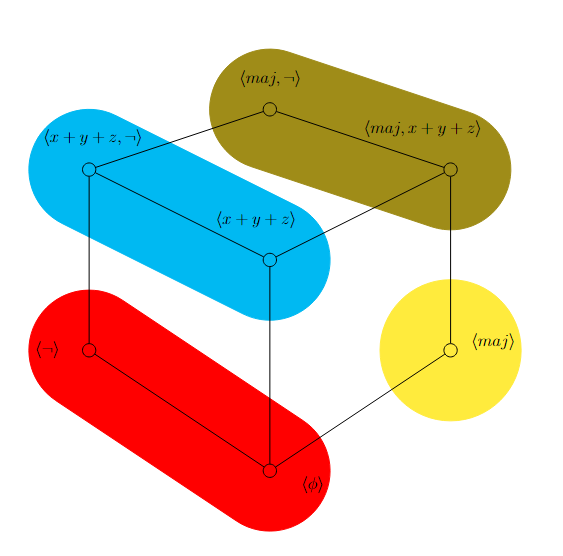}
    \caption{Clones on 2 elements with no TS polymorphism grouped by pp-construction class}
    \label{fig:reducedlattice}
\end{figure}

\begin{prop}
    For $\s D$ a Boolean structure, either $\s D$ has a $TS$-polymorphism or ${\s D}$ is pp-equivalent to one of the following:
    \begin{enumerate}
        \item 3\sat
        \item 2\sat
        \item 3\text{LIN}2
        \item $K_2$.
    \end{enumerate}
\end{prop}
\begin{proof}
Figure \ref{fig:reducedlattice} shows a part of Post's lattice, with five classes highlighted. Inspecting Post's lattice will reveal that the classes of algebras we include in Figure \ref{fig:reducedlattice} are those with no TS polymorphism. Note that
    
    \[\pol(2\sat)=\ip{\maj},\;\pol(\mathbb{F}_2)=\ip{ x+y+z},\; \pol(K_2)=\ip{\maj,\neg}\]

    \[\ip{\emptyset}=\pol(3\sat), \;\ip{\neg}=\pol(\{0,1\},\{(x,y,z): \neg(x=y=z)\})\]
    
    \[\pol(\mathbb{L})=\ip{ x+y+z,\neg},\; \pol(K_2,\{0\})=\ip{\maj,x+y+z}\] where $\maj(x,y,z)\in\{0,1\}$ is the unique value repeated among $x,y,z$, and $\bb L$ is $\bb Z/2\bb Z$ equipped with all affine relations which are preserved by $\neg$. 
    
    For a structure on a $2$-element set, being a core means that no constant function is a polymorphism. Expanding by all unary singleton relations kills off automorphisms, in the sense that if $\s D$ is $\s E$ with all single unary relations, then $\pol(\s D)$ is the set of elements of $\pol(\s E)$ satisfying $f(x,x,...,x)=x.$ Note that all of these structures are cores, so \[(K_2,\{0\})\equiv_{\text{pp}}K_2, \quad\mathbb{L}\equiv_{\text{pp}}(\mathbb{L},\{0\})\equiv_{\text{pp}}\mathbb{F}_2,\] 
    \[\mbox{ and } 3\sat\equiv_{\text{pp}}(\{0,1\},\{(x,y,z): \neg(x=y=z)\}).\]

    Therefore, the four highlighted classes in Figure \ref{fig:reducedlattice} exhaust the pp-construction classes among structures without TS polymorphisms.
\end{proof}

In fact, these are exactly the pp-reduction classes, as will follow from our results. In Section \ref{sec: provable} we will show that the structures $\mathcal{D}$ with $K_{\mathcal{D}}$ provable from $\zf$ are exactly those with a TS polymorphism. Thus, Figure \ref{fig:reducedlattice} shows the possible polymorphism algebras of structure with non-trivial compactness principles and their pp-equivalence classes. Kat\'ay, T\'oth, and Vidny\'ansky's theorem characterizes the compactness principles of maximal strength; in this case those which are pp-equivalent to $3\sat$. So, by Theorem \ref{tool1}, we only have to analyze the three intermediate cases of 2\sat, $\bb F_2$, and $K_2$ to understand all compactness principles for Boolean structures. We will analyze these cases in Section \ref{sec: boolean analysis}. In particular, we show they are all distinct and that $K_{K_2}$ is strictly weaker than the others. This answers a question posed by K{\'a}tay, T{\'o}th, and Vidny{\'a}nszky.

\subsection{Characterizing provable compactness principles} \label{sec: provable}
In this section, we will find all $K_{\mathcal D}$ that are provable from \zf. Note that this is only non-trivial when $\zf$ is consistent, which is what we shall assume in the following part.\par

We will begin by defining a class of relatively simple structures with sufficient symmetry so that a proof of their compactness principles can be carried out in \zf. They are the width-1 structures, by which we mean a structure whose satisfaction problem can be computed by keeping track of only facts about one variable with $k$ update rules, for some fixed $k\geq0$. Such structures have a nice algebraic characterization given in a classic work of Dalmau and Pearson \cite{dalmau1999closure}. We then show these structures correspond exactly to the \zf-provable compactness principles. One direction was already shown in \cite{RorabaughTardifWehlau}. We include a direct proof for the reader's convenience.

\begin{dfn} An instance $\mathcal X$ of $\mathcal D$ is \textbf{arc-consistent} if there is a function $U: \mathcal X \to \mathcal P(\mathcal D)$ such that 
\begin{enumerate}
    \item For each $x \in \mathcal X, U(x) \neq \emptyset$
    \item If $a \in U(x)$ and there is some tuple $(x_1, \cdots , x_n) \in \mathcal R^{\mathcal X}$ with $x_i=x,$ then there is some tuple $$(a_1, \cdots , a_n) \in \mathcal R^{\mathcal D} \cap U(x_1) \times \cdots \times U(x_n)$$ so that $a_i=a.$
\end{enumerate}

 \end{dfn}
Recall that a polymorphism $f$ is totally symmetric, or TS, if it only depends on the set of inputs (not the order or multiplicity). In this case, if the arity of $f$ is at least $|D|$, then $f$ can be interpreted as a function $f^*:\mathcal{P}(\mathcal{D})\setminus\{\varnothing\}\to D$ via \[f^*(A)=f(a_1,...,a_n)\] here $a_1,..., a_n$ is any enumeration of $A$. 

\begin{definition}
    $\mathcal{D}$ is width-1 if and only if every arc-consistent instance of $\mathcal{D}$ has a solution.
\end{definition}
Dalamau and Pearson only give a proof of this theorem for finite instances, but it works for infinite instances as well. We include a sketch for completeness.
\begin{theorem}[\zf] \label{thm: dp} A finite relational structure $\mathcal{D}$ is width-1 if and only if $\mathcal{D}$ has totally symmetric polymorphisms of arbitrarily high arity.
\end{theorem}
\begin{proof}
  Let's work in $\zf$. First, we prove the forward direction. Let $\mathcal D=(D,R_1,...,R_n)$ be width-1 constraint library and consider the instance \[\mathcal X=(\mathcal P(D)\setminus \emptyset,R_1^X,...,R_n^X)\] defined by \[R_i^X(A_1,...,A_{a_i}) :\lra\; (\forall j, \forall x_j\in A_j)(\exists x_i\in A_i \mbox{ for each } i\neq j)\; R(x_1,...x_{a_i}).\] The structure $\mathcal X$ is arc-consistent as witnessed by the identity map and therefore admits a solution $s:\mathcal X\to\mathcal D$. Now fix $n\geq|D|$ and consider $f:\mathcal D^n\to \mathcal D$ via $f(a_1,...,a_n)=s(\{a_1,...,a_n\})$. Then $f$ is a TS polymorphism of arity $n$ by construction of this $\mathcal X$.\par
  
  Now, for the converse suppose that $\s F$ is a family of TS polymorphisms of $\s D$ with arbitrarily large arity. One small subtlety is that we would like consider structures like $\bb F_2$ with infinitely many relations. Since $D^{\s P(D)}$ is finite, we can assume without loss of generality that each polymorphism $f\in\s F$ induces the same set map $f^*$ as discussed above. Fix any arc-consistent instance $\mathcal X$ of $\mathcal D$, and let $U:X\rightarrow \s P(D)$ witness arc-consistency. Let $f^*:\mathcal P(D)\setminus\{0\}\to\mathcal D$ be the common set map in $\s F$ and define $s:\mathcal X\to\mathcal D$ by $s=f\circ U$. We'll check that $s$ is a solution.

    Suppose $R^X(x_1,..., x_n)$. Fix $f\in \s F$ with arity at least $n|D|$. For each $a_i\in U(x_i)$ there is $\bar a\in R^{\s D}\cap \prod_i U(x_i)$ with $\pi_i(\bar a)=a_i$. So, we can form an $n\times n|D|$ matrix $M$ where the $i^{th}$ row enumerates $U(x_i)$ and each column is in $R^{\s D}.$ Say $M$ has rows $M_1,..., M_n$. Then, $f(M_i)=f^*(U(x_i))=s(x_i)$, and since $f$ is a polymorphism, $R^{\s D}(f(M_1),..., f(M_n))$. Thus, $R^{\s D}(s(x_1),...,s(x_n))$ as desired.
\end{proof}

\begin{theorem}[Con(\zf)]
    For a finite relational structure $\mathcal{D}$, the following are equivalent:
    \begin{enumerate}
        \item $\mathcal{D}$ is width-1 
        \item $K_{\mathcal{D}}$ is provable in \zf.
    \end{enumerate}

\end{theorem}

\begin{proof}
    $(1) \Rightarrow (2)$. Let $\mathcal{D}$ be any width-1 finite relational structure. For any instance $\mathcal{X}$ of $\mathcal{D}$, we will define a maximal witness to arc-consistency. 

    For $f: \s X\rightarrow \s P(\s D)$, say that a value $d\in \s D$ is good for $x\in \s X$ in $f$ if $d\in f(x)$ and for any constraint $\vec e\in R^\s X$ with $e_i=x$, $d$ can be extended by values from the $f(e_j)$'s to some $\vec a\in R$. More formally, $d$ is good for $x$ in $f$ if
    \[(\forall \mbox{ relations }  R^{\s D})(\forall \vec e\in R^{\s X})\;e_i=x\rightarrow (\exists \vec a\in R^{\s D}\cap \Pi_{j} f(e_j) ) \;a_i=d.\]
    
    We define a sequence of functions $F^{\mathcal X}_m: \mathcal X \to \mathcal P(\mathcal D) $ inductively:
    \begin{enumerate}
        \item $F^{\mathcal X}_0(x) = D$

        \item For $m\geq0$, 
       \[F^{\mathcal X}_{m+1}(x)=\{d\in D: d\mbox{ is good for }x\mbox{ in }F^{\s X}_m\}.\]

        \item $F^{\mathcal X}_{\infty}(x) = \cap_{i\in \mathbb N} F^{\mathcal X}_i(x).$
    \end{enumerate}
    Note that $F^\mathcal{X}_{m+1} (x)\subseteq F_m^\mathcal{X}(x).$ Since $F_0^\mathcal X(x)$ is finite, there exists some $M$ such that $m\geq M \implies F_m^\mathcal X(x) =F_M^\mathcal X(x) = F_\infty ^\mathcal{X}(x).$

    \begin{lemma}
        Every $d\in F^{\s X}_\infty(x)$ is good for $x$ in $F^{\s X}_\infty(x)$.
    \end{lemma}
    \begin{proof}
        As before let $M_j$ be such that $m \geq M_j$ implies $F_{M_j}^{\mathcal{X}}(e_j) = F_m^{\mathcal{X}}(e_j) = F_{\infty}^{\mathcal{X}}(e_j)$ for all $m\geq M_j$, and let $M$ be the max of the (finitely many) $M_j$. 
        
        Fix $d\in F^{\s X}_{\infty}(x)$. Since $d\in F^{\s X}_{M+1}$, $d$ is good for $x$ in $F^{\s X}_M(x)$. So, for any $\vec e\in R^{\s X}$ with $e_i=x$, there is some $\vec a\in R^{\s D}\cap \Pi_j F^{\s X}_M(e_j)=R^{\s D}\cap \Pi_j F^{\s X}_{\infty}(e_j)$. Thus, $d$ is also good for $F^{\s X}_{\infty}$.
    \end{proof}

    \begin{lemma}
        For any $m\geq0$ and $x\in X$, there is a finite $A\subseteq X$ such that $x\in A$ and $F^{\mathcal{X}}_m(x)=F^{\mathcal{A}}_m(x)$, where $\mathcal{A}$ is the induced substructure on $A$.
    \end{lemma}
    \begin{proof}
        We will proceed by induction on $m$.
        
        Base Case: When $m=0$, $F_{0}^{\mathcal{A}}(x)=D$ for any $A\subseteq X$, thus the claim holds trivially.

        Induction Step: Suppose now that for every $x\in X$ there is finite $A\subseteq X$ with $F^{\mathcal{X}}_m(x)=F^{\mathcal{A}}_m(x)$. Note that for any $x\in A\subseteq X$ we have $F^{\mathcal{X}}_{m+1}(x)\subseteq F^{\mathcal{A}}_{m+1}(x)$, so it suffices to find finite $A$ with $F^{\mathcal{A}}_{m+1}(x)\subseteq F^{\mathcal{X}}_{m+1}(x)$.
        
        For each $v\in F_{m}^{\mathcal X}(x)\setminus F_{m+1}^{\mathcal X}(x)$, 
        there exists a relation $R$ of $\s D$ and some $\vec e \in  R^{\mathcal X}$ with $x=e_j$ satisfying that for each $\bar{a}\in R$ with $a_j=v$ we have $\bar{a}\notin\prod_{k}F^{\mathcal{X}}_m(e_k)$. Since there are only finitely many such $v$, we can pick $\bar{e}^v$ for each $v\in F_{m}^{\mathcal X}(x)\setminus F_{m+1}^{\mathcal X}(x)$ with $\bar{e}^v\in R^{\mathcal{X}}_v$ but $v$ cannot be extended to $\bar{a}\in R_v$. Again there are only finitely many $v, k$ so by induction, we can pick a finite $A^v_k\subseteq X$ with $F^{\mathcal{A}^v_k}_m(e^v_k)=F^{\mathcal{X}}_m(e^v_k)$. Let $$A=\{x\}\cup\bigcup_{v,k} A^v_{k}$$
        and note that for each $v\notin F^{\mathcal{X}}_{m+1}(x)$, we have $e^v_k\in A,F^{\mathcal{A}}_m(e^v_k)\subseteq F^{\mathcal{A}^v_k}_m(e^v_k)=F^{\mathcal{X}}_m(e^v_k)$ so that $F^{\mathcal{A}}_m(e^v_k)=F^{\mathcal{X}}_m(e^v_k)$, so $v\notin F^{\mathcal{A}}_{m+1}(x)$ by construction of $\bar{e}^v$. Thus $F^{\mathcal{A}}_{m+1}(x)=F^{\mathcal{X}}_m(x)$, proving the lemma.
    \end{proof}
    This lemma then extends to $F_\infty^\mathcal X$ since it is the intersection of all $F_m^\mathcal X.$

    \begin{lemma}
        If every finite substructure has a solution, then $F_\infty^\mathcal X(x) \neq \emptyset$ for all $x.$
    \end{lemma}

    \begin{proof}
    Suppose towards a contradiction that there is $x \in \mathcal X$ such that $F_\infty^\mathcal X(x)= \emptyset$.
    Then for some $m\in\mathbb{N}$ we must have $F^{\mathcal{X}}_m(x)=\varnothing$ since $D$ is finite. By the previous lemma, there is a finite $A\subseteq X$ such that $F^{\mathcal{A}}_m(x)=\varnothing$. Note now that $\mathcal{A}$ is a finite substructure of $\mathcal{X}$, but there is no solution to $\mathcal{A}$, which is impossible.
    \end{proof}
    So, if $\s X$ is an arc-consistent instance of $\s D$, then $F_\infty^\mathcal X$ is an arc-consistency witness for $\s X$. Thus by Theorem \ref{thm: dp}, $\s X$ has a solution. This proves $K_{\mathcal{D}}$ in \zf.

    $(2)\Rightarrow(1)$. Suppose that $\mathcal D$ is not width-1. The fourth author showed that \zf+DC proves there is a Borel instance $\mathcal X$ of $\mathcal D$ with every finite subinstance solvable but without a Baire measurable solution \cite[Thm 5.7]{thornton2022algebraicapproachborelcsps}. Shelah showed \zf+DC+``all sets are Baire measurable'' is consistent relative to \zf\cite{shelah1984can}. Thus, in Shelah's model, $\mathcal X$ is a counterexample to $K_{\s D}$, so ZF+DC does not prove $K_{\s D}$.
\end{proof}

During the drafting of this article, Tardif also claimed the above result. However, Tardif's construction relies on measure theory and an unpublished result of Solovay \cite{tardif}. It is unclear if his proof requires an inaccessible cardinal. Our proof also has another consequence.  K\'atay, T\'oth, and Vidny\'anszky asked how the relations between compactness principles change under DC. It follows from our result that the provable structures remain the same.

\begin{corollary}[Con(\zf)]
    $\zf\vdash K_{\mathcal D}$ if and only if \zf+DC $\vdash K_{\mathcal D}$
\end{corollary}

\subsection{2-coloring, 2\sat, and Linear Equations} \label{sec: boolean analysis}

From the analysis of Post's lattice and results from the previous section, we know that there are at most 4 nontrivial compactness principles for structures on 2 elements: $K_{K_2}$, $K_{2\sat}$, $K_{3\sat}$, and $K_{3\text{LIN}2}$. From K\'{a}tay, T\'{o}th, and Vidny\'{a}nszky, we know that $K_{3\sat}$ is strictly stronger than the other three. We will show that $K_{K_2}$ is strictly weaker than the rest, and that $K_{2\sat}$ does not imply $K_{3\text{LIN}2}.$ Thus, there are exactly $4$ nontrivial principles, and only the relation between $K_{2\sat}$ and $K_{3\text{LIN}2}$ remains mysterious. Note that $K_{K_2}$ is equivalent to $\ac(2)$ (choice for 2-element sets) \cite{jech2008axiom}.

\begin{theorem}[Con(\zf)] \label{thm: boolean structures}
    $\zfa+K_{K_2}$ does not prove $K_{\mathbb{F}_2}$ or $K_{2\sat}$.
\end{theorem}\begin{proof}
    First, we will construct a model of ZFA where $K_{K_2}$ holds but neither $K_{\mathbb{F}_2}$ nor $K_{2\sat}$ does. To do so we consider a structure $(\mathbb{B},\leq)$ which consists of a Boolean algebra $\mathbb{B}=(B,0^{\mathbb{B}},1^{\mathbb{B}},\cap^{^{\mathbb{B}}},\cup^{^{\mathbb{B}}},\cdot^{c^\mathbb{B}})$ and a linear order $\leq$ on its underlying set such that $0^{\mathbb{B}}$ is a $\leq$-minimum and $1^{\mathbb{B}}$ is a $\leq$-maximum and which is generic for such structures (i.e., it is the Fra\"iss\'e limit of finite linearly ordered Boolean algebras). Then by taking a model of ZFA with $B$ the set of atoms, we construct a permutation model using the automorphism group $G$ of the structure $(\mathbb{B},\leq)$, and show it has the desired properties.

    Let $\mathcal{K}$ be the class of structures $\mathbb{A}=(A,0^{\mathbb{A}},1^{\mathbb{A}},\cap^{^{\mathbb{A}}},\cup^{^{\mathbb{A}}},\cdot^{c^\mathbb{A}},\leq_A)$ where $A$ is a finite set, $(A,0^{\mathbb{A}},1^{\mathbb{A}},\cap^{^{\mathbb{A}}},\cup^{^{\mathbb{A}}},\cdot^{c^\mathbb{A}})$ is a non-trivial Boolean algebra, and $(A,\leq_A)$ is a total order satisfying $\min_{\leq_\bb A}\bb A=0^\bb A$ and $\max_{\leq_{\bb A}}\bb A= 1^A$.
    
    We will construct $(\mathbb{B},\leq)$ as the Fra\"iss\'e limit of the class $\s K$. So $(\bb B, \leq)$ satisfies the following properties:
    \begin{enumerate}
        \item For any non-trivial finite Boolean algebra $\mathbb{B}_0\in \s K$ there is an embedding $(\mathbb{B}_0,\leq_0)\to(\mathbb{B},\leq)$, i.e. an injective function $\mathbb{B}_0\to\mathbb{B}$ which preserves the Boolean algebra structures as well as the linear order.
        \item For any two finite ordered subalgebras $(\mathbb{B}_0,\leq_0),(\mathbb{B}_1,\leq)\subseteq(\mathbb{B},\leq_1)$, and any isomorphism $f:(\mathbb{B}_0,\leq_0)\to(\mathbb{B}_1,\leq_1)$, there is an automorphism $g:(\mathbb{B},\leq)\to(\mathbb{B},\leq)$ which extends $f$, meaning $g(a)=f(a)$ for all $a\in\mathbb{B}_0$.
    \end{enumerate}
    By Fra\"iss\'e's theorem, this limit exists and has the desired properties as soon as $\s K$ satisfies the joint embedding and amalgamation properties. In fact, the joint embedding property follows from the amalgamation property: the 2-element Boolean algebra with its only compatible linear order embeds into any non-trivial Boolean algebra, so amalgamating two structures over this common substructure proves the joint embedding property.

    Suppose that we have three linearly ordered finite Boolean algebras, $(A,\leq_A),(B,\leq_B),$ and {$(C,\leq_C)$} and two embeddings $f:C\hookrightarrow B,g:C\hookrightarrow A$. We wish to find a structure $(D,\preceq_D,\leq_D)$ and embeddings $F:A\hookrightarrow D,G:B\hookrightarrow D$ such that the diagram commutes:\\
    \begin{center}\begin{tikzcd}
    & A \arrow[dr, hook, dashed, "F"] & \\
    C \arrow[ur, hook, "f"] \arrow[dr, hook, "g"] & & D \\
    & B \arrow[ur, hook, dashed, "G"] &
    \end{tikzcd}\\ \end{center}

    We can build a push-out of Boolean algebras $D=A\oplus_C B$ (see Definition \ref{dfn:boolean algebra pushout} in Appendix \ref{appendix: WEI}). Let $F, G$ be the natural maps from $A$ and $B$, respectively, into the push-out. We have a linear orders on $F[A]\subseteq D$ and $G[B]\subseteq D$ given by pushing forward the orders on $A$ and $B$. These agree on the intersection $F[A]\cap G[B]=F[C]$ by the assumption that $f$ and $g$ are embeddings in $\s K$. By Lemma \ref{lem:useful lemma 1} in Appendix \ref{appendix: WEI}, this means we can extend these pushforwards to a linear order on $D$ which makes the diagram commute in the category of linearly ordered Boolean algebras. Thus the class $\s K$ has the amalgamation property.

    Now we may consider the Fra\"iss\'e limit $(\mathbb{B}, \leq^B)$ where $\bb B$ is a countable atomless boolean algebra, with operations and relation $\sqcup^B,\sqcap^B,\cdot^B,{0^B}^B,{1^B}^B$, a relation $x\sqsubseteq^B y:\lra x\sqcap^B y=x$, and $\leq^B$ is a linear order. For notational convenience, we shall drop the superscript (unless there is a risk of ambiguity) and also write $x\triangle y=(x\sqcup y)\sqcap (x\sqcap y)^c$ for symmetric difference.\par
    Fix a model of ZFA whose set of atoms is (the underlying set of) $(\mathbb{B},<)$. Let $G$ be the group of permutations on this set given by the automorphisms of $(\mathbb{B},<)$, and let $\mathcal{M}$ be the permutation model associated to $G$.

    First, we claim that $\mathcal{M}\models K_{K_2}$. More generally, we will show that every set can be linearly ordered in $\s M$; therefore, choice for finite sets holds. By Theorem \ref{thm: wei} in Appendix \ref{appendix: WEI}, for $\bb A, \bb C\subseteq \bb B$ subalgebras,
    \[\fix(\bb A\cap \bb C)=\ip{\fix(\bb A),\fix(\bb B)}\] (that is, the stabilizer of the intersection of subalgebras is generated by the union of their stabilizers). It follows that, for any set $X\in \s M$, there is a smallest (in terms of containment) subalgebra $\supp(X)\subseteq \bb B$ so that $\fix(\supp(X))\subseteq \sym(X)$. The rest is a standard argument, which will outline. As in \cite{jech2008axiom}, the map $X\mapsto \supp(X)$ is hereditarily symmetric, and for any set $X\in \s M$, the map
    \[\iota: X\rightarrow [\bb B]^{<\infty}\times X/\sym(X)\]
    \[S\mapsto (\supp(S), \sym(X)\cdot S)\] is hereditarily symmetric and injective. In $\s M$, $X/\sym(X)$ is in bijection with an ordinal and $[\bb B]^{<\infty}$ has a linear order (the lexicographic order induced by $\leq$), so $X$ has a linear order. 

    Now we will show $\mathcal{M}\models\neg K_{2\sat}$. Consider the instance $\mathcal{X}$ of $K_{2\sat}$ whose set of variables is $\mathbb{B}$, and whose relations are given by $R_{\rightarrow}(x,y)$ if and only if $x\leq y$, and $R_\neg(x,x^c)$ for all $x\in\mathbb{B}$.

    This is indeed an HS instance of 2\sat: the underlying set is HS since it is exactly the set of atoms, and therefore the whole set is stabilized by any element of $G$. Furthermore, since elements of $G$ are automorphisms and therefore preserve the Boolean algebra structure as well, and since any Boolean algebra homomorphism must preserve the $x\leq y$ relation, we see the relations for this instance are also preserved by any element of $G$. In the ground model that any the characteristic function of any ultrafilter on $\mathbb{B}$ yields a solution.
    
    We want to show that, on the other hand, no hereditarily symmetric solution exists for $\mathcal{X}$. Suppose towards the contrary that $f:\mathcal{X}\to\{0,1\}$ is a solution with $\fix(E)\subseteq \sym(f)$ for some finite subalgebra $\bb E\subseteq \bb B$. Suppose $\mathbb{E}$, has order atoms $E'=\{e_i:i<n\}$, and note that any automorphism fixing $E$ must also fix $\bb E$. 
    I claim that we can find $x_0,x_1,x_2,x_3\in \bb B$ so that $x_i\sqcap x_j=0$ for each $i\not=j$, $\bigsqcup_{i=0}^3 x_i=1$, and so that there are $\sigma_1,\sigma_2,\sigma_3\in \fix(\bb E)$ with \[\sigma_1(x_0)=x_1, \sigma_2(x_0)=x_2^c,\mbox{ and } \sigma_3(x_3)=x_3^c.\] Let $B_4$ be the 4-atom Boolean algebra with atoms $y_0,y_1,y_2$ and $y_3$. Let $h, g$ be the canonical algebra embeddings of $B_4$ and $\bb E$ respectively into $\bb A:=B_4\oplus \bb E$. Then for $i\not=j$, $\ip{h(y_i), g[\bb E]}\cap \ip{h(y_j),g[\bb E]}=g[\bb E].$ By Lemma \ref{lem:useful lemma 2} in Appendix \ref{appendix: WEI}, any choice of linear orders on $\ip{h(y_i), g(\bb E)}$ that agree on $g(\bb E)$ extends to a linear order on $\bb A$. In particular there is a linear order $\leq$ on $\bb A$ with the following properties:
    \begin{enumerate}
        \item $h(y_0)\leq h(y_0)^c, h(y_1)\leq h(y_1)^c, h(y_2)\geq h(y_2)^c, h(y_3)\geq h(y_3)^c$
        \item for all $i,j$, $(\ip{h(y_i), g[\bb E]},\leq)\cong (\ip{h(y_j), g[\bb E]},\leq)$ via isomorphisms that fix $g[\bb E]$
        \item $g: (\bb E, \leq)\rightarrow(\bb A, \leq)$ is an embedding of linearly ordered algebras.
    \end{enumerate} By the universality of Fra\"iss\'e limits, there is an embedding $\phi:\bb A\rightarrow\bb B$ with $\phi\circ g=\id_\bb E$. Set $x_i=\phi\circ h(y_i)$. By ultrahomogeneity,  the isomorphisms from item $(2)$ above lift to automorphisms of $\bb B$ in $\fix(\bb E)$ as desired. 
    
    Now, by symmetry of $f$, we have that \[f(x_0)=f(x_1)=1-f(x_2)=1-f(x_3).\] If $f(x_0)=0$, then $f(x_2)=f(x_3)=1$ and $f(x_3^c)=0$. But $x_2\sqsubseteq x_3^c$, so this is a contradiction. On the other hand, if $f(x_0)=1$, then $f(x_1^c)=0$, but $x_0\sqsubseteq x_1^c$, again we have a contradiction. Thus $\s X$ has no symmetric solution.
    
    \par

    We now proceed to show that $\mathcal{M}\models\neg K_{\mathbb{F}_2}$. Consider the instance $\s Y$ of $\bb F_2$ with domain $B$ and linear equations with variables in $\bb B$ of the form $1^B=1, 0^B=0$, and
    \[x_1+x_2+...+x_n=y\] whenever $x_1\,\triangle\, ...\,\triangle\,x_n=y$. Just as before, $\s Y$ is invariant under $G$, so $\s Y\in \s M$. And as before, an ultrafilter gives a solution to $\s Y$ in the ground model.
    
    Now, we show that $\s Y$ has no symmetric solution. Suppose towards contradiction that $\s Y$ has a symmetric solution $f$, and let $\bb E$ be the algebra generated by a support for $f$. As before, we can find a partition $x_0,x_1,x_2,x_3$ in $\bb B$ with $x_1, x_2^c, x_3^c\in \fix(\bb E)\cdot x_0.$ Since $x_i\,\triangle\, x_i^c=1^B$, $f(x_i)+f(x_i)=1.$ It follows that $f(x_0)=f(x_1),$ and $f(x_2)=f(x_3)$, so \[f(1^B)=f(x_0\, \triangle \,x_1\,\triangle\, x_2\,\triangle\, x_3)=f(x_0)+f(x_0)+f(x_2)+f(x_2)=0\] But, this is a contradiction.
    
\end{proof}

To separate $K_{\bb F_2}$ from $K_{2\sat}$, we can use a more familiar model.

\begin{thm}[Con(\zf)]
    $\zf+K_{2\sat}$ does not prove $K_{\bb F_2}$
\end{thm}
\begin{proof}
    Rorabaugh, Tardif, and Wehlau showed that $K_{2\sat}$ follows from the Order Extension Principle (which is elsewhere called Szpilrajn's Lemma) \cite{RorabaughTardifWehlau}. Roughly, we can use the usual proof of logical compactness to reduce $K_{\s D}$ to its restriction to quotients of powers of $\s D$ along filters. If $\s X$ is a quotient of a power of $2\sat$, then we can solve $\s X$ by first extending $R_{\rightarrow}^{\s X}$ to a (quasi) linear order $\leq $ and then setting $f(x)=0$ if there is $y$ with $x<y$ and $R_{\neg}(x,y).$

    Now, as in the work of Felgner and Truss \cite{FelgnerTruss}, we can consider the symmetric model $\s M$ associated to the generic pair $(\bb B, \leq)$ where $\bb B$ is a boolean algebra and $\leq$ is linear order extending the partial order $\sqsubseteq$ from $\bb B$. The order extension principle is known to hold in $\s M$, so it suffices to show $\s M\not \vDash K_{\bb F_2}$.

    Again, we use the natural interpretation of $\bb B$ as an instance of $\bb F_2$. We define a system of $\bb F_2$-linear equation $\s X$ with variables $X=\{x_b: b\in \bb B\}$ and equations

    \[x_{0^B}=0,\quad x_{1^B}=1,\quad \mbox{ and }\]
    \[\sum_{i=1}^n x_{b_i}=x_a\mbox{ whenever }b_1\,\triangle\, ...\,\triangle\, b_n=a\] where $x\,\triangle\, y=(x\sqcup y)\sqcap(x\sqcap y)^c.$ Again, an ultrafilter gives a solution in the ground model, so $\s M$ believes that every finite subinstance of $\s X$ has a solution. Let us check that $\s X$ has no symmetric solution. 
    
    Suppose towards contradiction that $f: X\rightarrow \{0,1\}$ is a symmetric solution with support $\bb A$. Without loss of generality, we may assume that $\bb A$ is a subalgebra. Since the order atoms of $\bb A$ sum to $1^B$, there is an order atom $a\in \bb A$ with $f(x_a)=1$. Fix $b\sqsubset a$ not in $\bb A$. As in Felgner and Truss's article, we can find are $b,c,d,d' \in \bb B$ so that \[c,b\sqcup c, d,d'\in \fix(\bb A)\cdot b\] and 
    \[b\sqcup d\sqcup d'=a,\quad b\sqcap d=b\sqcap d'=0^B.\] So, since $\fix(f)\supseteq \fix(\bb A)$, \[f(x_b)=f(x_c)=f(x_{b\sqcup c})=f(x_b)+f(x_c).\] It follows that $f(x_b)=0$. On the other hand, $f(x_b)=f(x_d)=f(x_{d'}).$ Thus, \[f(x_b)=f(x_{b\sqcup d\sqcup d'})=f(x_a)=1.\] This is a contradiction.
\end{proof}

$\,$

\section{Compactness principles for directed cycles}
In this section, we construct an infinite chain and an infinite anti-chain of compactness principles ordered by implication. The anti-chain we will consider is $\{K_{C_p}:p\text{ prime}\}$, the compactness principles for directed cycles of prime lengths. In fact, we show the incomparability of their compactness principles in a strong sense: for each $p$ we will exhibit a model $\mathcal{M}_p$ of ZFA which simultaneously satisfies $K_{C_q}$ for every $q\neq p$ but nevertheless fails $K_{C_p}$. Utilizing this structure of directed cycles, we also will build a strictly increasing infinite chain $\{K_{\mathcal{D}_n}:n\geq 2\}$ related to the disjoint union of cycles $\bigsqcup_{i=2}^n C_i$ and show that $K_{\mathcal{D}_n}$ is equivalent to $\ac(n)$.

\begin{dfn}
    For a directed graph $G=(V,E)$, a \textbf{path} in $G$ is a sequence $p= (v_1,e_1,v_2,e_2,...,e_{n-1},v_n)$ of vertices $v_i$ and edges $e_i$ so that, for each $i$, $e_i=(v_i,v_{i+1})$ or $e_i=(v_{i+1},v_i)$. A \textbf{directed path} is a path so that $(v_{i},v_{i+1})=e_i$ for all $i$. The \textbf{algebraic length} of a path $p$ is $m-k$ where $m$ is the number of indices $i$ such that $(v_i,v_{i+1})=e_i$ and $k$ is the number where the reverse $(v_{i+1},v_i)=e_i$ holds. We say $G$ is\textbf{ weakly connected} if there is a path between any two vertices. A \textbf{component} of $G$ is a maximal non-empty weakly connected induced subgraph.
\end{dfn}

Observe that the coloring of a weakly connected component is completely determined by its value on any vertex, in the sense that if $G$ is connected and $f,g:G\to C_n$ are homomorphisms to $C_n$ with $f(v)=g(v)$ for some $v\in G$, then $f=g$. 

\begin{prop}[\zf] \label{lem: cyclic homs} For every $n$, the restriction of $K_{C_n}$ to weakly connected graphs holds. Moreover, each connected graph admits either $0$ or exactly $n$ homomorphisms to $C_n$.
\end{prop}
\begin{proof} Let $G=(V,E)$ be weakly connected and fix a $v\in V$. Suppose we want to build a homomorphism $f$ with $f(v)=i$ for some $i\in C_n$. For each $u\in G$, there is some path $p$ from $u$ to $v$ in $G$, so we must assign $u$ the value $i+j \mod n$ where $j$ is the algebraic length of $p$. Thus, there is at most $1$ homomorphism with $f(v)=i$. If this does define a homomorphism, then any choice of $f(v)$ will work giving us $n$ homomorphisms.

If this is not a homomorphism, such a failure will be witnessed by some subgraph with finite distance from $v$ (namely, two paths of different algebraic length modulo $n$), and therefore there is a finite subgraph of $G$ with no homomorphism to $C_n$.

\end{proof}

Before proceeding, we first record a technical lemma:
\begin{lemma}
    For every $n\geq 2$, the following are equivalent over \zf:
    \begin{enumerate}
        \item $K_{C_n}$
        \item For every graph $G=(V,E)$ of disjoint directed $n$-cycles, there is a homomorphism $h:G\to C_p$.
    \end{enumerate}
\end{lemma}
\begin{proof}
Of course $(1)$ implies $(2)$. Suppose $(2)$ holds. Let $G$ be any graph where each component has a homomorphism to $C_n$. Define an auxiliary graph $G^*= (V^*, E^*)$ where
\[V^* :=\{h\in C_n^\s C: \s C \mbox{ is a component of }G, h\mbox{ is a homomorphism}\}\]
\begin{align*}E^*(f,g)\;:\lra\;& (\exists v)\;f(v)+1=g(v) \\\lra\;& \dom(f)=\dom(g)\mbox{ and }(\forall v\in \dom(f))\;f(v)+1=g(v).\end{align*} By the argument above, each component of $V^*$ is an $n$-cycle, so there is a homomorphism $H: G^*\rightarrow C_n$. Now we can define a homomorphism on $G$: for $v\in G$, set
\[f(v)=i \;\lra\; (\exists g\in V^*)\;g(v)=v\mbox{ and }H(g)=0.\]
\end{proof} 

We now proceed to the construction showing that $K_{C_p}$ and $K_{C_q}$ are pairwise independent over ZFA. We must specify a set of atoms and a group $G$ of permutations on this set, and we will then take $\mathcal{M}_p$ to be the symmetric model associated to $G$.

\begin{definition}
    Let $\mathcal{A}$ be a countable set of atoms, split $A$ into disjoint $p$-cycles, and define the group $G$ to be the group which rotates these cycles. More concretely, let the atoms be $A=\{a_{i,j}:i\in\mathbb{N},j\in \bb Z/p\bb Z\}$ with $G$ being the subgroup of the permutation group of $A$ so that $\sigma\in G$ if and only if for each $i\in\mathbb{N}$ there is some $n_i\in\bb Z/p\bb Z$ so that for all $j\in \bb Z/p\bb Z$ we have $\sigma(a_{i,j})=a_{i,j+n_i}$.
\end{definition}
A key point in our proof willbe that every element of $G$ has order $1$ or $p$.

\begin{theorem} \label{thm: antichain}
    For any primes $p,q$, the permutation model $\mathcal{M}_p$ of HS sets constructed from $G$ satisfies $\mathcal{M}_p\vDash K_{C_q}$ if and only if $p\neq q$.
\end{theorem}
\begin{proof}
    First let us show $\mathcal{M}_p\nvDash K_{C_p}$. Consider the instance of $C_p$ whose underlying set is exactly the set $\mathcal{A}$ of atoms, and whose graph structure is given by how we defined the group $G$.

    More precisely, let $A_0,A_1,\dots$ be the pairwise disjoint sets of size $p$ from Definition 38 (so $\mathcal{A}=\bigcup_{i\geq0}A_i$). Each $A_i$ was given the structure of a $p$-cycle, so let this digraph be the instance of $C_p$.

    First note that this is an HS$_G$ instance of $C_p$: the underlying set is the set of atoms, which is stabilized by all permutations, and because the permutations only rotate the cycles $A_i$, the edge relation is preserved by any permutation as well.

    By Theorem \ref{tool4}, it remains to show there is no HS$_G$ solution. Suppose to the contrary that there is. Then, there is an HS homomorphism to $C_p$ of this digraph, i.e. a homomorphism which is fixed as soon as we fix a finite set of atoms. But this is impossible: given any finite set $E\subseteq\mathcal{A}$, $E$ contains vertices from only finitely many of the $A_i$, and therefore after fixing the vertices from those $A_i$, we can take some $A_N$ not among those $A_i$ and consider the element of $G$ which rotates just $A_N$. This permutation does not preserve the homomorphism.

    Now let us show when $q$ is a prime with $p\neq q$ that $\mathcal{M}_p\vDash K_{C_q}$. By the previous lemma, it suffices to show every HS instance of $C_q$ composed of disjoint $q$-cycles has a solution. Fix such a directed graph, $\s G$. Let $\mathcal{C}_1$ and $\mathcal{C}_2$ be two $q$-cycles of $\mathcal{G}$. First, note that if $x,y\in \mathcal{C}_1$ and $\pi\in G$ with $\pi(x)=y$ then $x=y$: since $\pi$ has order $1$ or $p$, $\pi^p=1\in G$, so $\pi^p(f_1)=f_1$. However, in this case it follows that $\pi\in \sym(C_1)$, and since $C_1$ is a q-cycle, $\pi^q(x)=x$ as well. Since $\gcd(p,q)=1$, we find $f_1=\pi(f_1)=g_1$.

    Similarly, if $x\in\mathcal{C}_1$ and $y,z\in\mathcal{C}_2$ and there are permutations $\pi,\sigma\in G$ with $\pi(x)=y$ and $\sigma(x)=z$, then the composition $\pi\sigma^{-1}$ sends $y,z$, so $f_2=g_2$.

    We deduce that for vertex $x$ of $\s G$, the orbit of $x$ (under $\sym(\s G)$) will meet each connected component of $\mathcal{G}$ in at most one vertex. Therefore the action by elements of $\sym(\s G)$ on $\mathcal{G}$ only permutes the cycles, i.e. connected components, of $\mathcal{G}$ (without rotations). So, if $\s G=(V,E)$ and we write $[v]$ for the orbit $\sym(\s G)\cdot v$, then $\s G'=(V/\sym(\s G), E')$ where $E'([u],[v])\lra E(u,v)$ is also a union of disjoint $q$-cycles. Let $f'$ be a ground model coloring of $\s G'$ and define $f:V\rightarrow C_q$ by $f(v)=f'([v])$. Then, $f$ is a homomorphism and invariant under $\sym(\s G)$. So, $f\in \text{HS}_G$.
\end{proof}

Even though the compactness principles for the directed cycles are incomparable in a strong sense, we can utilize the structure similar to taking their disjoint union (which we define in Section 1.1) to get a chain.\par
Following the order sketched in the introduction, we now prove that the construction of the infinite increasing chain of compactness principles works as expected:
\begin{theorem}\label{thm: finite choice}
    The following are equivalent over \zf:
    \begin{enumerate}
        \item $K_{\mathcal{D}_n}$
        \item $\ac(n)$
        \item $\bigwedge_{i=2}^nK_{C_i}$
    \end{enumerate}
\end{theorem}
\begin{proof}
$(1)\Rightarrow(3):$ Each $C_i$ is a pp-definable substructure of $\s D_n$.  
$(2)\Rightarrow(1)$: For $\s X$ an instance of $\s D_n$, each component of $\s X$ has at most $n$ solutions whose images land in a cycle of minimal length. If each component has a solution, then we can pick a solution using $\ac(n)$. Otherwise, as in Lemma \ref{lem: cyclic homs}, there is a finite obstruction. 

$(3)\Rightarrow(2)$: We proceed by induction. For the base case, $n=2$, suppose $F$ is a family of sets size at most 2. We want to construct an instance of $C_2$ whose solutions corresponds to choice functions. Define $\mathcal{X}=(X, R^{\s X})$ where \[X=\{x:(\exists S\in F)\;|S|=2\wedge x\in S\}\]\[R^{\mathcal{X}}=\{(x,y)|x,y\in S\in F\wedge x\neq y\}\] Concretely, it assigns to each 2-element set a 2-cycle, with each element corresponding to a vertex. Any finite subinstance of $\mathcal{X}$ is solvable since it has no cycle of odd length, so $\mathcal{X}$ admits to a solution function $s$, which maps the elements to either $0$ or $1$. Then \[f(S)=x:\Leftrightarrow S=\{x\}\lor (x\in S\wedge s(x)=0)\] defines a choice function for $F$.

Now assume inductively that $\bigwedge_{i=2}^n C_i$ is equivalent to $\ac(n)$, and assume that $\bigwedge_{i=2}^{n+1} C_i$ holds. Then, $\ac(n)$ holds by induction. Given a family $F$ of sets size at most $n+1$, it suffices to consider the sets with exactly $n+1$ elements. Moreover, we have a choice function $c$ for the family
\[F'=\{S\setminus \{x\}: x\in S\in F\}.\] For each $S\in F$, consider the directed graph $G_S=(S,E_S)$, where $E_S=\{(a,b):a\in S\wedge b=c(S\setminus \{a\})\}$ (that is, each element points to the element given by $c$). Even though the connected components $G_S$ are not necessarily directed cycles, we do have the following:
\begin{lemma}
    All weakly connected directed graphs $G$ with $n\geq 2$ vertices so that each vertex points to exactly one other vertex, have a homomorphism to a directed cycle $C_m$ with $m\leq n$. Furthermore, $m=n$ if and only if $G$ is a cycle.
\end{lemma}
\begin{proof}
    This statement is straightforwardly true when there are 2 vertices, so assume inductively that it holds true for all graph with n vertices. For any such graph G with $n + 1$ vertices, if all vertices are pointed to then G forms an $n + 1$ cycle, else there is a vertex $v$ not pointed to by any vertex. Denote the vertex pointed by $v$ as $w$ and consider the subgraph of $G$ excluding $v$ and the edge $(v, w)$; by inductive hypothesis it admits a homomorphism $h$ to a directed cycle $C_m$ for $m\leq n<n+1$. Then the function \[
    H(x)=
    \begin{cases}
        c:\Leftrightarrow (c,h(w))\text{ is an edge in }C\text{ if }x=v\\
        h(x)\text{ else}
    \end{cases}\]
    is a homomorphism from $G$ to $C_m$.
\end{proof}
If for some $S$, $G_S$ breaks up into more than one connected component, then we can use $\ac(n)$ to choose a connected component, then to choose a vertex from that component.  So, we have a choice function for $F''=\{S\in F: G_S\mbox{ is not connected}\}.$ 

For $i\leq n$, we can define a choice function for \[F_i=\{S\in F: G_S\mbox{ is connected and has a homomorphism to }C_i\}\] as follows. By $\bigwedge_{i\leq n+1} K_{C_i}$, for each $i$, there is homomorphism from $f:\bigcup_{S\in C_i}G_S\rightarrow C_i$. Then, for each $S$, for some $j\in C_i$, $0<|f\inv[j]\cap S|<n+1$. Say $j_s$ is the least such $j$. So, using $\ac(n)$, we can choose an element from each $f\inv[j_s]$, where $j_s$ is as small as possible. Finally, we can patch the choice functions for $F''$ and each of the $F_i$'s together to get a choice function for $F$. 
\end{proof}

When combining this with the incomparability between $K_{C_p}$ and $K_{C_q}$, we derive that in ZFA, $K_{\mathcal{D}_m}$ is strictly weaker than $K_{\mathcal{D}_n}$, and therefore also $\ac([m])$ weaker than $\ac(n)$, if there is a prime $p$ in between $m$ and $n$ with $m<n$. Hence,
\begin{corollary}
    $\{K_{\mathcal{D}_p}:p\text{ prime}\}$ is a strictly increasing chain of compactness principles in the provability order.
\end{corollary}

\printbibliography

\newpage
\appendix
\section{Weak elimination of imaginaries} \label{appendix: WEI}

 Let $\s K$ be the class of finite structures $(\bb A, \leq)$ where
 \begin{enumerate}
     \item $\bb A$ is a Boolean algebra (in the signature $\sqcup, \sqcap, \cdot^c,0,1$)
     \item $\bb \leq$ is a linear order
     \item $0^\bb A<1^\bb A$ and $(\forall x\in \bb A)\; 0^\bb A\leq x\leq 1^\bb A.$
 \end{enumerate} Let $(\bb B,\leq)$ be the Fra\"iss\'e limit of $\s K$ (which was shown to exist in the text). Our goal is to show that $(\bb B,\leq)$ has weak elimination of imaginaries, i.e.~that for any finite subalgebras $\bb A, \bb C\subseteq \bb B$
\[\fix(\bb A\cap \bb C)=\ip{\fix(\bb A),\fix(\bb C)}.\] Throughout, we write $\fix(\bb A)$ for the pointwise stabilizer of $\bb A$ in $\aut(\bb B, \leq)$. For any sets ${X,Y\subseteq \aut(\bb B, \leq)}$ or ${X,Y\subseteq \bb B}$, we write $\ip{X,Y}$ for the group or subalgebra they generate.

First, let's recall some basic theory of Boolean algebras.
\begin{dfn}
    For $\bb A$ a Boolean algebra, let $\at(\bb A)$ be the set of order theoretic atoms of $\bb A$:
    \[a\in \at(\bb A)\;:\lra \;a\in \bb A, a\not=0,(\forall b\in \bb A)\;b\sqsubset a\rightarrow b=0.\]
    For $x,y\in \bb A$, say
    \[x\;\mathbf{cuts}\;y\;:\lra\; x\sqcap y \not=0\;\&\;x\sqcap y\not=y.\]
    \[x\perp y\;:\lra\;(x\mbox{ cuts }y \;\&\; y\mbox{ cuts }x) .\] We pronounce $x\perp y$ as $x$ and $y$ are \textbf{independent}.
    For $\bb C\subseteq \bb A$ a subalgebra, define
    \[x\perp_{\bb C}y\;:\lra \;(\forall c\in \at(\bb C))\, x,y\mbox{ cut }c\rightarrow  (x\sqcap c\perp y\sqcap c)\] and by convention declare that $x\perp_{\{0,1\}} y$ iff $x\perp y$.
    
    For subalgebras $\bb A,\bb C, \bb D\subseteq \bb B$, define
    \[\bb A\perp_{\bb D}\bb C\;:\lra\; \bb A\cap \bb C=\bb D\mbox{ and }(\forall a\in \bb A,c\in \bb C)\;a\perp_{\bb D} c.\]
\end{dfn}
Note that, $x\perp y$ if and only if $x\sqcap y,x\sqcap y^c, x^c\sqcap y$, and $x^c\sqcap y^c$ are all nonzero.

For $\bb A$ any finite Boolean algebra, it's well known that $\s P(\at(\bb A))\cong \bb A$ via the map $s\mapsto \bigsqcup s$. And, any map $f: \bb A\rightarrow \bb C$ factors through this isomorphism: for any $f\in \hom(\bb A, \bb C)$, there is function $F:\at(\bb C)\rightarrow \at(\bb A)$ so that
\[f(x)= \bigsqcup F\inv[\{a\in \at(\bb A): a\sqsubseteq x\}].\] So, we can amalgamate finite Boolean algebras by considering fibred products of finite sets.

\begin{dfn}\label{dfn:boolean algebra pushout}
    For set maps $f: A\rightarrow C$ and $g: B\rightarrow C$,
    \[A\otimes_C B:=\{(a,b)\in A\times B: f(a)=g(b)\}.\]

    For $\bb A$ and $\bb C$ any finite algebras with a common subalgebra $\bb D$, define $f: \at(\bb A)\rightarrow \at(\bb D)$ by
    \[f(a)=d\;:\lra a\sqsubseteq d.\] Define $g: \at(\bb B)\rightarrow \at(\bb D)$ similarly.
    
    We then have a Boolean algebra
    \[\bb A\oplus_\bb D \bb C= \s P\bigl(\at (\bb A)\otimes_{\at(\bb D)} \at(\bb C)\bigr)\]
    and an embedding $f_\bb A\in \hom(\bb A, \bb A\oplus_{\bb D}\bb C)$
    \[f_\bb A:\bb A\rightarrow \bb A\oplus_\bb D \bb C\]
    \[f_\bb A(x)=\{(a,c)\in \at(\bb A)\otimes_{\at(\bb D)}\at(\bb C): a\sqsubseteq x\}.\] (An embedding $f_\bb C\in\hom(\bb C, \bb A\oplus_{\bb D}\bb C)$ is defined mutatis mutandis.)
\end{dfn}

It is clear from construction that $f_\bb A$ and $f_\bb C$ are embeddings and that any pair of maps $g_{\bb A},g_{\bb C}$ that agree on $\bb D$ factor through the $\bb A\oplus_\bb D\bb C$. That is, we have defined a push-out for diagrams of Boolean algebras. See also \cite[Chapter 5]{DaveyPriestley}.

The following lemma tells us $\perp_\bb D$ is the free-ness relation associated this kind of amalgamation:

\begin{lem}\label{lem:useful lemma 1}
    For finite Boolean algebras $\bb A,\bb C\subseteq \bb B$, with $\bb A\cap\bb C=\bb D$, the following are equivalent:
    \begin{enumerate}
        \item $\ip{\bb A,\bb C}\cong \bb A\oplus_{\bb D} \bb C$ (via the map induced from $\id_\bb A$ and $\id_{\bb C}$)
        \item $\bb A\perp_{\bb D}\bb C$
        \item For any $a\in \at(\bb A),c\in \at(\bb C)$, $a\perp_{\bb D} c$.
    \end{enumerate}
\end{lem}
\begin{proof}
    The equivalence of $(2)$ and $(3)$ is straightforward algebra. To see that $(1)$ implies $(3)$, note that if $a\in \at(\bb A), c\in\at(\bb C)$ and $a,c$ both cut $d\in \at(\bb D)$, then $f_\bb A(a)$ is a horizontal cross section of $\{(x,y)\in \at(\bb A)\times \at(\bb C): x,y\sqsubseteq d\}$ and $f_\bb C(c)$ is a vertical cross-section. These cut each other.

    Finally, to see that $(2)$ implies $(1)$, we just need to check that the induced map $g: \bb A\oplus_\bb D\bb C\rightarrow \ip{\bb A, \bb C}$ is injective. It suffices to check this on atoms. Say we have $a,a'\in \at(\bb A), c,c'\in \at(\bb A)$ and $d,d'\in \at(D)$ with $a,c\sqsubset d$ and $a',c'\sqsubset d'$. Since $\bb A\perp_\bb D \bb C$, $a\sqcap c\not=0$ and $a'\sqcap c'\not=0$. Suppose  $g(\{(a,c)\})=g(\{(a',c')\}).$ Then, $a\sqcap c=a'\sqcap c'$. Since neither is $0$, we must have that $d=d'$. If one of $a,a',c,c'=d$, then clearly $a=a'$ and $c=c'$. If none of these atoms are in $\bb D$, then $a\perp c$ and $a'\perp c'$ so, either $a=a'$ and $c=c'$ or $a\sqcap c, a'\sqcap c'$ are disjoint nonzero element of $\ip{\bb A,\bb C}$. Thus, in either case $a=a'$ and $c=c'$, so $g$ is injective. 
\end{proof}

The next lemma is helpful for amalgamating orders. Its proof is an easy but tedious exercise, which we leave for the reader:

\begin{lem} \label{lem:useful lemma 2}
    If $X_1,..., X_n$ are sets with $X_i\cap X_j=Y$ for all distinct $i,j\in [n]$, $\leq_Y$ is a partial order on $Y$, and for each $i$ $\leq_i$ is a partial order on $X_i$ so that $\leq_i\res Y=\leq_Y$, then the transitive closure of $\bigcup_i \leq_i$ is a partial order on $\bigcup_{i=1}^n X_i$. In particular, there is a linear order on $\bigcup_i X_i$ extending each $\leq_i$.
\end{lem}

We prove our main theorem in several steps. Want to write every $\rho\in \fix(\bb A\cap \bb C)$ as a product of automorphisms fixing either all of $\bb A$ or all of $\bb C$. First, we'll show that we can do this assuming that $\bb C$ and $\rho[\bb C]$ are independent of $\bb A$, independent of each other, and in a canonical order relative to $\bb A$.

\begin{dfn}
    For $\bb A,\bb C, \bb D\subseteq \bb B$, write 
    \[\bb A\leq_\bb D \bb C\;:\lra \;(\forall a\in \bb A, c\in \bb C)\;c<a \rightarrow (\exists d\in \bb D) \;c\leq d\leq a.\]

    For an order $\leq$ on a set $A$ and a bijection $f: A\rightarrow B$, define the \textbf{pushforward} $f_*(\leq)$ on $B$ by $x \;f_*(\leq)\; y:\lra f\inv(x)\leq f\inv(y)$
\end{dfn}
So, $\bb A\leq_\bb D \bb C$ if elements of $\bb A$ come before elements of $\bb C$ in gaps between elements of $\bb D$.

\begin{lem} \label{lem: normal form}
    Suppose $\bb A, \bb C_0, \bb C_1, \bb D\subseteq \bb B$ are finite ordered subalgebras so that 
    \begin{enumerate}
        \item $\bb A\perp_\bb D \bb C_0$, $\bb A\perp_\bb D\bb C_1$, and $\bb C_0\perp_\bb D \bb C_1$
        \item $\bb A\leq_\bb D \bb C_0,\bb C_1$
        \item $(\exists \rho\in \fix(\bb D))\;\rho[\bb C_0]=\bb C_1.$
    \end{enumerate} Then, there is $\sigma\in \ip{\fix(\bb A),\fix(\bb C_0)}$ so that $\sigma\res \bb C_0=\rho\res \bb C_0.$
\end{lem}

\begin{proof}
    First, set $\s X= \ip{\bb C_0, \bb A}$ and $\s Y=\ip{\bb C_1, \bb A}, \s Z=\ip{\bb C_0, \bb C_1}$ as a algebras, and \[\s W=\s X\oplus_\bb A \s Y.\] Since $\s Z\cong \bb C_0\oplus_\bb D \bb C_1$, we have algebra embeddings $f_\s X, f_\s Y, f_\s Z$ from $\s X ,\s Y$, and $\s Z$ respectively to $\s W$ so that \[\im f_\s X\cap \im f_\s Y=f_\s X[\bb A]=f_\s Y[\bb A],\] and $f_\s Y\res\bb A= f_\s X\res \bb A.$ Similarly for the other pairs of maps: $f_\s Z,f_\s X$ agree at $\bb C_0$ and $f_\s Y, f_\s Z$ agree at $\bb C_1.$

    Now, since $\bb C_0, \bb C_1\perp_\bb D \bb A$ and $\bb C_0\cong \bb C_1$, we have a Boolean algebra isomorphism $g: \s X\rightarrow \s Y$ which is the identity on $\bb A$ and $\rho$ on $\bb C_0$. We can order $\s Y$ to make this an ordered algebra isomorphism. Set $\preceq = g_*(\leq),$ and define
    \[\bb X=(\s X, \leq), \quad \bb Y =(\s Y, \preceq),\quad  \bb Z=(\s Z, \leq).\] We can find an order on $\s W$ which agrees with (the pushforward of the) orders from $\bb X$, $\bb Y$, and $\bb Z$. More formally, let $\leq_0=(f_{\s X})_*(\leq),$ $\leq_1=(f_{\s Y})_*(\preceq)$, and $\leq_2=(f_{\s Z})_*(\leq)$. We want to apply Lemma \ref{lem:useful lemma 2} first to $\im(f_{\s X})$ and $\im(f_{\s Z}),$ then to $\im(f_{\s Y})$ and $\im(f_{\s X})\cup\im(f_{\s X})$. Note that $\im(f_{\s X})\cap \im(f_{\s Z})=f_{\s X}(\bb C_0)$, and $\leq_0, \leq_2$ clearly agree here since they are pushforwards of the same order along identical maps, so we have a partial order $\leq'$ on $\im(f_{\s X})\cup \im(f_{\s Z})$ which agree with $\leq_0$ and $\leq_2$. We can compute $\bigl(\im(f_{\s X})\cup \im(f_{\s Z})\bigr)\cap \im(f_{\s Y})=f_{\s Y}[\bb C_1\cup \bb A].$ If $a,b\in \bb A$, $a\preceq b$ iff $a\leq b$, so $\leq'$ and $\leq_1$ agree in this case. If $a,b\in \bb C_1$, then $a\preceq b$ iff $\rho\inv(a)\leq \rho\inv(b)$ iff $a\leq b$. So, again $\leq'$ and $\leq_1$ agree. Finally, if $b\in \bb C_1$ and $a\in \bb A$, then, since $\bb A\leq_\bb D \bb C_1$, $b\leq a$ iff there is $d\in \bb D$ with $b\leq d\leq a$. Since $\rho\in \fix(\bb D)$, this is equivalent to $g\inv(b)=\rho\inv(b)\leq d\leq a$. So, again $\leq'$ agrees with $\leq_1$. Thus, we can find an order $\leq$ on $\s W$ which agrees with $\leq_0, \leq_1$ and $\leq_2$ as desired.

    We now have two ways to embed $\bb X$ into $\bb B$. The identity map is an embedding. By the genericity of Fra\"iss\'e limits, the identity map extends to an embedding $\psi: \bb W\rightarrow \bb B$, and then we have an embedding $\psi\circ g$. Observe that \[\psi\circ g\res \bb A=\id\mbox{ and }\psi\circ g\res \bb C_0=\rho.\] Again by genericity, there is an automorphism $\sigma\in\aut(\bb B)$ that connects these two embeddings: \[(\psi\circ g)\res \bb X=\sigma\res \bb X.\] In particular, $\sigma\in \fix(\bb A)$ and $\sigma\res \bb C_0=\rho\res \bb C_0$ as desired.

\end{proof}

Now, we need to maneuver $\bb C$ and $\rho(\bb C)$ to be free of $\bb A$ over $\bb D$. The lemma below lets us take one step in the right direction:

\begin{lem}[Step lemma] \label{lem: step lemma}
    For finite subalgebras $\bb A,\bb C_0,\bb C_1\subseteq \bb B$, there is $\sigma\in \fix(\bb A)$ so that
    \begin{enumerate}
        \item $\sigma[\bb  C_1]\cap \bb C_0\subseteq \bb A\cap \bb C_1$
        \item $\sigma[\ip{\bb C_1\cup \bb A}]\leq_\bb A \ip{C_0\cup \bb A}$
        \item $\sigma [\ip{\bb C_1\cup\bb A}]\perp_\bb A \ip{ \bb C_0\cup \bb A}$
    \end{enumerate}
\end{lem}
\begin{proof}
    Let $\s X= \ip{\bb A\cup \bb C_0}$ and $\s Y=\ip{\bb A\cup \bb C_1}$. Set $\s Z=\s X\oplus_\bb A\s Y$, and let $f_\s X: \s X\rightarrow \s Z$, $f_\s Y: \s Y\rightarrow \s Z$ be the associated algebra embeddings.
    
    So far $\s Z$ is just an algebra, but we can equip it with an ordering. Let \[\leq_0 := (f_\s X)_*(\leq)\quad\quad \leq_1 := (f_\s Y)_*(\leq)\] be the pushforwards of the inherited orders onto $\im (f_\s X)$ and $\im f_\s (Y)$ respectively. Since $\im (f_\s) X\cap (\im f_\s Y)=f_\s X[\bb A]$ and $f_\s X\res \bb A=f _\s Y\res \bb A$, Lemma \ref{lem:useful lemma 2} says that there is a partial order $\preccurlyeq$ on $\s Z$ which restricts to $\leq_0$ and $\leq_1$ on their domains.

    Importantly, $\preccurlyeq$ also satisfies that if $a,a'\in \bb A$ are so that there is no $a''\in \bb A$ with $a\leq a''\leq a'$, then the $\preccurlyeq$-interval $(f_\s X(a),f_\s X(a'))$ is a union of two chains--one from $\bb C_0$ and one from $\bb C_1$. We can extend  $\leq_1$ to a linear order on $\s Z$ by declaring that $f_{\s Y}[\bb C_1]$ comes before $f_\s X[\bb C_1]$ in each $\bb A$-gap, i.e.:
    \[x\leq y \;:\lra \; x\preccurlyeq y\mbox{ or }(x\in \im f_\s Y \mbox{ and }x\not\preccurlyeq y).\] In particular, for $f_\s Y[\bb C_1]\leq_{f_\s Y[\bb A]} f_\s X[\bb C_0].$

    Now, by genericity of $\bb B$, there is an ordered algebra embedding $\psi: (\s Z,\leq)\rightarrow \bb B$ so that $\psi\circ f_\s X=\id_\s X.$ And, $\psi\circ f_\s Y \res\bb A=\psi\circ f_\s X\res \bb A=\id_\bb A$. By the definition of Fra\"iss\'e limit, $\psi\circ f_\s Y$ extends to an automorphism $\sigma\in \fix(\bb A)$. And, since $\psi$ is an embedding, $\sigma[\s Y]=(\psi\circ f_\s Y)[\s Y]$ and $\s X=(\psi\circ f_{\s X})[\s X]$ stand in the same relation. Namely, $(2)$ and $(3)$ of the lemma are satisfied. Lastly, \[\sigma[\bb C_1]\cap \bb C_0\subseteq \psi[\im f_\s Y\cap \im f_\s X]=\psi\circ f_{\s Y}[ \bb A]=\bb A,\] and since $\sigma\in \fix(\bb A)$, this means that if $x\in \bb A\cap \sigma(\bb C_1),$ $x=\sigma\inv(x)\in \bb C_1$. So, the intersection is as claimed.
\end{proof}

It is routine to check that if $\bb C_1, \bb C_0$, $\bb A$, and $\sigma$  are as in the step lemma, then $\sigma$ doesn't undo any progress. That is, if $x\in \bb C_1$ and $y\in \bb C_0\cup \bb A$ satisfy
\[x\leq y\mbox{ or }(\exists d\in \bb C_0\cap \bb A)\;x\geq d>y\] then $\sigma(x)\leq y$ or $(\exists d\in \bb C_0\cap \bb A)\; \sigma(x)\geq d>y$. Similarly, if $x,y$ satisfy $x\perp_\bb D y$, then $\sigma(x)\perp_{\bb D} y$.

Iterating the above lemma over and over, first fixing $\bb A$ then $\bb C$ then $\bb A$ and so on, we can get $\bb C$ into the desired position:

\begin{lem}[Bubble sort lemma] \label{lem: bubblesort}
Fix any finite subalgebras $\bb A, \bb C, \bb X\subseteq \bb B$, and set $\bb D=\bb A\cap \bb C$. There is a $\sigma\in \ip{\fix(\bb A),\fix(\bb C)}$ so that 
\[\sigma[\bb X]\perp_{\bb D} \bb C\quad\mbox{ and }\quad \sigma[\bb X] \leq_\bb D \bb C. \]

And, the same relations hold between $\sigma(\bb X)$ and $\bb A$.
\end{lem}
\begin{proof}
    By the step lemma, we can assume that $\bb X \cap \bb C=\bb X\cap \bb A=\bb D$. For $\sigma\in \aut(\bb B, \leq)$, define $S_\star(\sigma)$ to be the set of witnesses to the failure $\sigma(\bb X)\leq_\bb D \bb C,\bb A$:
    \[S_\star(\sigma)=\{ (x,y)\in \bb X\times(\bb A\cup \bb C): \sigma(x)>y \;\&\; \neg(\exists d\in \bb D)\; \sigma(x)\geq d>y\}.\]
    Similarly, $S_\dagger(\sigma)$ is the set of witnesses to $\sigma(\bb X)\not\perp_\bb D \bb A, \bb C$
    \[S_\dagger=\{(x,a)\in \bb X\times(\at(\bb A)\cup \at(\bb C)): \neg(x\perp_{\bb D} a)\}.\]
    
    Pick some $\sigma\in \fix(\bb A, \bb C)$ so that $|S_\star(\sigma)|+|S_\dagger(\sigma)|$ is minimal. Suppose towards contradiction that one of $S_\star(\sigma)$ or $S_\dagger(\sigma)$ is nonempty. We'll use the step lemma to show that we can find $\tau\in \ip{\fix(\bb A),\fix(\bb B)}$ reducing $|S_\dagger(\sigma)|+|S_\star(\sigma)|$.

    Say that $S_\star(\sigma)$ is nonempty. We have some $x\in \sigma[\bb X]$ and $a\in \bb A\cup \bb C$ with no $d\in \bb D$ so that $x\geq d>c$. We can assume that $c$ is the largest such element of $\bb A\cup \bb C$ and, without loss of generality, $c\in \bb C$. Then, there is no $a\in \bb A$ with $x \geq a>c$. The step lemma (along with the comments following it) applied to $\bb C_0=\bb C$ and $\bb C_1=\sigma[\bb X]$ says there is $\tau\in \fix(\bb A)$ so that $S_\star(\tau\circ \sigma)\subseteq S_\star(\sigma)$, $S_\dagger(\tau\circ \sigma)\subseteq S_\dagger(\sigma),$, and $(\tau\circ \sigma)(\bb X)\leq_\bb A \bb C$. In particular, $(\tau\circ \sigma)(x)\leq c$. But this contradicts the minimality of $\sigma$.

    Now say that $S_\dagger(\sigma)$ is nonempty. For each atom $d\in \at(\bb D)$ which is not an atom of $\bb A$ or $\bb C$, define a graph $G=(V,E)$ with
    \[V=\{a\in \at(\bb A)\cup \at(\bb C): a\sqsubseteq d\}\]
    \[E(a,a')\;:\lra \;a\not=a'\mbox{ and } a\sqcap a'\not=0.\] Then, $G$ is connected since if $C$ is a component of $G$ $\bigsqcup C$ is in both $\bb A$ and $\bb C$. And, if $E(a,c)$, either $a\sqcap c\not=a$ or $a\sqcap c\not=c$ since otherwise $a=c\in \bb D$. For any $x\in \bb B$ with $x\sqcap d\not=0,d$, the we can apply the step lemma twice to find $\rho\in \ip{\fix(\bb A),\fix(\bb C)}$ so that if $x$ cuts $a\in V$ then $\rho(x)\perp a$ and $\rho(x)\perp a'$ for every neighbor $a'$ of $a$. By connectivity, we can iterate the step lemma to find $\rho$ so that $\rho(x)$ cuts every element of $V$. Fix some $(x,a)\in S_\dagger(\sigma)$. Then $a\not\in \bb D$ (otherwise $x\perp_\bb D a$ automatically). From the argument just outlined applied to the $\bb D$-atom containing $a$, we can iterate the step lemma to find $\tau\in \ip{\fix(\bb A),\fix(\bb C)}$ so that $S_\star(\tau\circ \sigma)\subseteq S_\star(\sigma)$, $S_\dagger(\tau\circ\sigma)\subseteq S_\dagger(\sigma)$, and $(x,a)\not\in S_\dagger(\tau\circ \sigma)$. Again, this contradicts minimality.

    Thus both sets are empty and $\sigma$ is as desired.
\end{proof}

Finally, we can prove weak elimination of imaginaries for $\bb B$.

\begin{thm}\label{thm: wei}
    For any finite subalgebras $\bb A,\bb C\subseteq \bb B$, 
    \[\fix(\bb A\cap \bb C)=\ip{\fix(\bb A),\fix(\bb C)}.\]
\end{thm}
\begin{proof}
    Fix $\rho\in \fix(\bb A\cap \bb C)$ and set $\bb D=\bb A\cap \bb C$. By the bubble sort lemma, there are some $\sigma,\tau\in \ip{\fix(\bb A),\fix(\bb C)}$ so that 
    \begin{enumerate}
        \item $(\sigma\circ\rho)[\bb C]\cap \bb A=\tau[\bb C]\cap \bb A=(\sigma\circ \rho)[\bb C]\cap \tau[\bb C]=\bb D$
        \item $(\sigma\circ \rho)[\bb C]\perp_\bb D \tau[\bb C], (\sigma\circ\rho)[\bb C]\perp_\bb D \bb A, \tau[\bb C]\perp_\bb D \bb A$
        \item $(\sigma\circ \rho)[\bb C]\leq_\bb D \bb A, \tau[\bb C]$ and $\tau[\bb C]\leq_\bb D\bb A$.
    \end{enumerate}
    By Lemma \ref{lem: normal form}, there is some $\nu\in \ip{\fix(\bb A),\fix(\bb C)}$ so that $\nu\res \bb \tau[\bb C]=\sigma\circ\rho\circ\tau\inv\res\tau [\bb C].$ Thus, 
    \[\nu\inv\circ \sigma\circ\rho\circ\tau\inv\in \fix(\tau[\bb C])=\tau \fix(\bb C)\tau\inv\subseteq\ip{\fix(\bb A),\fix(\bb C)} \] Since $\nu,\sigma, $ and $\tau$ are all in $\ip{\fix(\bb A),\fix(\bb C)},$ so it $\rho$. The other containment is obvious.
\end{proof}

\section{ZF results} \label{appendix: zf}

Two of the main results of this paper--Theorem \ref{thm: boolean structures} and Theorem \ref{thm: antichain}--are stated as theorems about $\zfa$. We will explain how to draw analogous conclusions in $\zf$.

First, we will briefly explain how to transfer our results about Boolean structures in $\zfa$ to results about $\zf$ using standard techniques. Permutation models admit embeddings into submodels of generic extensions, and our results pass through the basic embedding theorems. We will state these embedding theorems below and sketch how they apply in our situation.

Throughout, we will assume that $M$ is a transitive model of $\zfa+\ac$ with set of atoms $A$. As usual, this assumption can be avoided by reflection arguments.

\begin{dfn}
    Say that $\s I\subseteq [A]^{<\infty}$ is a generating family for $G$ if:
    \begin{enumerate}
        \item For any $E\in [A]^{<\infty}$ there is an $F\in \s I$ with $E\subseteq F$
        \item $\s I$ is $G$-invariant, i.e.~ if $E\in \s I$ and $\pi\in G$, then $\pi(E)\in \s I$.
    \end{enumerate}
    Note that $\s I$ is in $HS^G$ under these conditions. 

    Say that $G$ has \textbf{least supports} if there is a generating family $\s I$ for $G$ so that, for any $E,F\in \s I$
    \[\fix(E\cap F)=\ip{\fix(E), \fix(F)}.\]
\end{dfn}

For instance, the groups $\aut(\bb B, \leq)$ as in Theorem \ref{thm: boolean structures} has least supports with the family of finite subalgebras of $\bb B$ as a generating family
\[\s I=\{\ip{E}: E\in [B]^{<\infty}\},\] and the groups $G_p$ as in Theorem \ref{thm: antichain} have least supports where $\s I$ is the set of unions of cycles.

\begin{thm}[{\cite[Theorem 6.1, 6.7]{jech2008axiom}}] \label{thm: embedding theorem}
    Fix any ordinal $\alpha\in \s M$ and any group $G\subseteq \aut(A)$. There is a generic extension $M[\s G]$ and a submodel $M^G[\s G]\vDash \zf$, and a map
    $\tilde \cdot: HS^G\rightarrow \s M^G[\s G]$ so that
    \begin{enumerate}
        \item $\tilde \cdot$ is an embedding, i.e.~it is injective and $\tilde x\in \tilde y$ if and only if $x\in y$
        \item If $x\in \tilde y$ for some $y$, then $x=\tilde z$ for some $z$.
        \item $\tilde \cdot$ is an isomorphism between $(\s P^\alpha(A))^{HS^G}$ and $(\s P^\alpha(\tilde A))^{\s M^{G}[\s G]}$ (in particular, it is onto)
        \item if $G$ admits least supports, then, there is an ordinal $\kappa\in \s M^G[\s G]$ so that for any $X\in \s M^{G}[\s G]$ there is some $Y\in HS^G$ and a one-to-one map $f: X\rightarrow \s P(\kappa)\times \tilde Y$ with $f\in M^{G}[\s G]$.  
    \end{enumerate}
\end{thm}

This theorem says that every permutation model of $\zfa$ embeds into a model of \zf, and the embedding can cover any desired segment of the universe. Of course the embedding cannot hit the entire universe, but if $G$ admits least supports, then the ``gap" can be filled in by a set of ordinals.

Roughly, one constructs $\s M[\s G]$ by adding, for each $a\in A$, a generic set of subsets of $\kappa$, $\tilde a\subseteq 2^\kappa$ for some large enough $\kappa$. Then, $G\ltimes S_\kappa$ acts naturally on the forcing poset $\bb P$ and so on the class of all $\bb P$-names. The submodel $\s M^{G}[\s G]\subseteq \s M[\s G]$ is the set of interpretations of hereditarily symmetric names, and every hereditarily symmetric element of $M$ has a hereditarily symmetric canonical name which gives the embedding. The reverse embedding using supports is a technical modification of the argument that the model in Theorem \ref{thm: boolean structures} can linearly order every set. The details are all contained in \cite{jech2008axiom}.

\begin{prop}[Con(\zf)]
    Analogous to Theorem \ref{thm: boolean structures}, 
    \[\zf+K_{K_2}\not\vdash K_{2\sat}, K_{\bb F_2}.\]
\end{prop}
\begin{proof}
    Let $G=\aut(\bb B, \leq)$ as in the proof of Theorem \ref{thm: boolean structures}. In that proof, we found instances $\s X$ and $\s Y$ of $K_{2\sat}$ $K_{\bb F_2}$ respectively with $\s X, \s Y\in \s P^4(A)$. Any solution to these instances must be in $\s P^\omega(A)$. So, find a $\zf$ model $\s M^G[\s G]$ as in Theorem \ref{thm: embedding theorem} with $\alpha=\omega$. Then, $\tilde{ \s X}$, $\tilde {\s Y}$ witness that 
    \[M^{G}[\s G]\not\vDash K_{2\sat}, K_{\bb F_2}.\] 

    But, the order property can be verified in $\s M^G[\s G]$ exactly as in \cite[Section 5.5]{jech2008axiom}. Any set $X\in\s M^G[\s G]$ embeds into $\s P(\kappa)\times \tilde Y$ for some $Y\in HS^G$. There is a linear order $\leq$ of $Y$, so we can order $\tilde Y$ by $\tilde \leq$ in $\s M^{G}[\s G]$. Thus, we can linearly order $\s P(\kappa)$ and $\s P(\kappa)\times \tilde Y$ lexicographically. Finally, we can pull back the order on $\s P(\kappa)\times \tilde Y$ to an order on $X$. Thus, 
    \[M^G[\s G]\vDash K_{K_2}.\]
\end{proof}

\begin{prop}[Con(\zf)]
    $\zf+K_{2\sat}\not \vdash K_{\bb F_2}$
\end{prop}
\begin{proof}
    The \zfa model we use comes from work of Truss and Felgner \cite{FelgnerTruss}. In their paper, they show that their argument for the Order Extension Principle also passes through the embedding, and so too does our counterexample to $K_{\bb F_2}$ (just as above).
\end{proof}

Unfortunately, our results about $K_{C_p}$ do not pass immediately through the embedding theorems. One can use more delicate forcing arguments to adapt them into a $\zf$ construction, but we will provide a different argument below. The proof of Theorem \ref{thm: finite choice} gives:

\begin{prop}
    For any prime $p$, \[\zf+K_{C_p}+\bigwedge_{n<p} \ac(n)\vdash \ac(p).\]
\end{prop}

\begin{cor}
    For any prime $p$, \[\zf+\bigwedge_{q\;\mathrm{prime}\;q\not=p}  K_{C_q} \not \vdash K_{C_p}\]
\end{cor}
\begin{proof}
    It follows from Truss's $A(Z,v)$ criterion \cite{Truss} that \[\zf+\bigwedge_{\{n: p\;\not | n\}} \ac(n)\not\vdash \ac(p).\] So, by the proposition above,
    \[\zf+\bigwedge_{\{n: p\;\not | n\}} \ac(n)\not\vdash K_{C_p}.\] But clearly 
    \[\zf+\bigwedge_{\{n: p\;\not | n\}} \ac(n)\vdash K_{C_q}\] for all primes $q\not=p$. Thus, $\zf+\bigwedge_{q\;\mathrm{prime}\;q\not=p}  K_{C_q} \not \vdash K_{C_p}$.
\end{proof}

\end{document}